\documentclass[smallextended,numbook,runningheads,nospthms]{svjour3}     
\smartqed  
\usepackage{graphicx}
\usepackage{amsmath, amsthm, amsfonts}
\usepackage{epstopdf} 
\usepackage{bm}

\usepackage[T1]{fontenc}  
\usepackage{type1ec}      
\usepackage{hyperref}
\usepackage{cleveref}
\usepackage{placeins}
\usepackage{wrapfig}
\usepackage{tikz-cd}
\DeclareMathOperator{\Diff}{Diff}
\DeclareMathOperator{\id}{id}
\DeclareMathOperator{\op}{op}
\DeclareMathOperator{\GL}{GL}
\DeclareMathOperator{\Lip}{Lip}
\DeclareMathOperator*{\argmin}{arg\,min}
\newcommand{\Ess}{\mathcal{S}}
\newcommand{\Eff}{\mathcal{P}}

\newcommand{\N}{\mathbb{N}}
\newcommand{\cL}{\mathcal{L}}
\newcommand{\wb}{\overline}
\newcommand{\cg}{\mathfrak{g}}

\journalname{BIT}

\newtheorem{theorem}{Theorem}[section]
\newtheorem*{theorem*}{Theorem}

\theoremstyle{plain}
\newtheorem{la}[theorem]{Lemma}
\newtheorem{prop}[theorem]{Proposition}

\numberwithin{equation}{section}

\theoremstyle{definition}

\newtheorem{numba}{}[section]

\theoremstyle{remark}
\newtheorem{rem}[theorem]{Remark}

\newtheorem{example}{Example}[theorem]

\begin{document}
\title{Deep neural networks on diffeomorphism groups for optimal shape reparameterization}
\thanks{
EC has received funding from the European Union’s Horizon 2020 research and innovation programme under the Marie Sk{\l}odowska-Curie grant agreement No 860124. EC would like to thank the Isaac Newton Institute for Mathematical Sciences, Cambridge, for support and hospitality during the programmes {\it Mathematics of deep learning} and {\it Geometry, compatibility and structure preservation in computational differential equations} where work on this paper was significantly advanced. This work was supported by EPSRC grant no EP/R014604/1.
}
\titlerunning{Deep learning of diffeomorphisms}        

\author{Elena Celledoni      \and
        Helge Gl\"ockner     \and
        J\o rgen~Riseth      \and
        Alexander Schmeding
}

\institute{
    Elena Celledoni \at
    Department of Mathematics, NTNU, Trondheim \\
    \email{elena.celledoni@ntnu.no}           
    \and
    Helge Gl\"ockner \at
    Institut f\"{u}r Mathematik, Universit\"at Paderborn\\
    \email{glockner@math.uni-paderborn.de}
    \and
    Jørgen Riseth
    \at Department of Numerical Analysis and Scientific Computing, Simula Research Laboratory
    \at Department of Mathematics, University of Oslo \\
    \email{jorgennr@simula.no}
    \and
    Alexander Schmeding \at
    Department of Mathematics, NTNU, Trondheim \\
    \email{alexander.schmeding@ntnu.no}
}

\date{\today}

\maketitle

\begin{abstract}
One of the fundamental problems in shape analysis is to align curves or surfaces before computing geodesic distances between their shapes.
Finding the optimal reparametrization realizing this alignment is a computationally demanding task, typically done by solving an optimization problem on the diffeomorphism group.
In this paper, we propose an algorithm for constructing approximations of orientation-preserving diffeomorphisms by composition of elementary diffeomorphisms.
The algorithm is implemented using PyTorch, and is applicable for both unparametrized curves and surfaces.
Moreover, we show universal approximation properties for the constructed architectures, and obtain bounds for the Lipschitz constants of the resulting diffeomorphisms.
\end{abstract}

\section{Introduction}
The geometric shape is of major importance in object recognition. The field of shape analysis gives a formal definition of the shape of an object, and provides tools to compare these shapes.
One way to represent objects mathematically is to consider their contours as parametric curves or surfaces defined on some compact domain $ \Omega \subseteq \mathbb{R}^n $.
Since the same shape may be outlined by different parametric curves that are equivalent under composition by some diffeomorphism, one typically have to solve an optimization problem on the group of orientation preserving diffeomorhpisms, \( \mathrm{Diff}^+(\Omega) \), to compute shape distances.

In this article, we propose an algorithm for solving optimization problems on $\mathrm{Diff}^+(\Omega)$ by composition of elementary diffeomorphisms.
Using concepts from deep learning, we restate the reparametrization problem such that it may be solved by training a residual neural network.
This approach admits an effective, unified framework for optimal reparametrization of both curves and surfaces.
We motivate the approach with results on infinite-dimensional Lie groups of diffeomorphisms by describing global charts for certain diffeomorphism groups, thus implying universal approximation properties.
Furthermore, we establish a priori estimates for the composition of multiple diffeomorphisms.
These estimates are of independent interest beyond the applications considered in the present paper.
Finally, we provide an implementation of the algorithm  based on the \texttt{PyTorch}-framework~\cite{paszke2019} and test it for several reparametrization problems.
The source code for our implementation of the algoritms is available at \url{https://github.com/jorgenriseth/funcshape.git}\footnote{This repository will be the subject of change following any feedback during the review process. At the point of publication this url will be replaced by one pointing to a fixed version.}

\subsection{Related work}
The methods proposed in this article were developed in the framework of shape analysis, but have their foundations in optimal control and deep neural networks. 
Our approach is a so-called \textit{learning-free} method, which does not attempt to draw conclusions to unseen instances based on features learned from big data sets, but rather attempt to find an alignment between a single pair of shapes.
As such, the method is conceptually similar to that of a \textit{deep image prior} \cite{ulyanov2018deep}, where untrained deep convolutional networks are used for image reconsunstruction tasks, such as denoising, based on a single image.

The neural network is structured to have inherently encoded some of the important properties of diffeomorphisms, such as invertibility.
To achieve this, we base the network structure on residual neural networks \cite{He}, where each layer is of the form $ x \mapsto x + f(x)$.
Here $f:\mathbb{R}^{d}\rightarrow \mathbb{R}^d$ depends on parameters, e.g.\ for so-called dense layers $f(x):=\sigma(Ax+b)$, where $A$ is a $d\times d$ matrix, $b\in\mathbb{R}^d$ and $\sigma$ is a scalar nonlinearity applied componentwise.

Residual networks may further be related to approximations of diffeomorphisms, by interpreting them as flows of ordinary differential equations \cite{haber2017,weinan2017proposal,Benning2019ode}.
Assuming the parameters to be time-dependent functions $A(t)$ and $b(t)$, each layer may be interpreted as a forward Euler discretization of the non-autonomous differential equation 
\begin{equation}
    \label{resnetEq}
\dot{x}(t)=f(x,t), \quad x(0)=x,\qquad f(x,t):=\sigma(A(t)x+b(t)),
\end{equation}
$$
    x\mapsto \psi_L,\quad \psi_k=\psi_{k-1}+hf(\psi_{k-1},t_k),\quad k=1,\dots, L, \quad \psi_0=x,
$$
where $h$ is the step-size often taken to be equal to $1$, and with $t_k=kh$.
For $f$ Lipschitz continuous in~$x$, the Euler method is known to converge on bounded time intervals to the solution of the differential equation as $h\rightarrow 0$.
A sufficient condition for the invertibility of the layer map is that $hf$ is Lipschitz continuous, with Lipschitz constant $\mathrm{Lip}(hf)<1$ (also called $1$-Lipschitz).
For activation functions $\sigma$ with bounded derivative, most layer types, e.g.\ dense layers or convolutional layers, can be made $1$-Lipschitz by appropriate scaling \cite{celledoni2021}. See for example \cite{behrmann2021} for a concrete strategy for choosing such a scaling.

The necessity of invertible layers in neural networks is not restricted to the ODE-interpretation of residual neural networks. For example, normalizing flows are a class of machine learning models that may be used to artificially generate data \cite{tabak2010}. Given $m$ independent observations of $n$ random variables one seeks an estimate of their underlying probability density $\rho(x)$. The approximation of the probability density is achieved by mapping $x$ to a new set of variables $y=\varphi ^{-1}(x)$ with known (for example normal) distribution $\mu(y)$.  
Then
$$ \rho(x)=J_y(x)\mu(y(x)),$$
where $J_y$ is the Jacobian of the map $\varphi ^{-1}$. The map $\varphi ^{-1}$ is built by means of a gradient ascent flow increasing the log-likelihood of the sample data with respect to its density.

While the evolution and growth of the field of deep learning  have primarily been driven by empirical results
throughout the last decade, there are multiple approximation theorems which provide a theoretical basis for the effectiveness of neural networks. For example, consider the space of one-layer neural networks of the form
\begin{equation}
    \psi: [0, 1]^n \to \mathbb{R}, \qquad \psi(x):= \sum_{j=1}^N\alpha_j\sigma(w_j^T x+b_j),
\end{equation}
with parameters $\alpha_j,b_j \in \mathbb{R}$,  $\mathbf{w}_j \in \mathbb{R}^n$, and some activation function $ \sigma: \mathbb{R} \to \mathbb{R} $. In 1989, Cybenko~\cite{cybenko1989} proved that such networks are dense in the space of continuous functions with respect to the supremum norm as $N \to \infty$, provided that $\sigma$ is continuous and satisfies
\[
    \sigma (t)\rightarrow 
    \left\{\begin{array}{ccc}
        1 & & t\rightarrow +\infty,\\
        0 & & t\rightarrow -\infty.
    \end{array}\right.
\]

In \cite{agrachev2009}, the authors provide a similar approximation result for diffeomorphisms isotopic to the identity, defined on a compact manifold $\Omega$. They show that such diffeomorphisms may be expressed as a finite composition of exponential maps of vector fields that have been scaled by smooth functions, as long as the vector fields belong to a bracket-generating family of functions. In \cite{agrachev2021}, this result is extended to more general maps on $\Omega$.

\subsection{Structure of the Article}
The
article is structured as follows: Section~\ref{background} provides background on shape analysis and the definition of the proposed deep-learning approach.
\Cref{sect:Diffeos} gives universal approximation results and the theoretical basis for the deep learning approach.
\Cref{sect:multcomp} describes a priori estimates of the $C^k$-distance between the identity diffeomorphism and the network $\varphi$, given in terms of elementary diffeomorphisms.
\Cref{sect:methods} provides details on how the proposed algorithm is implemented.
\Cref{sec:numres} shows the results of applying the algorithm to test examples involving both parametric curves, surfaces and +nFinally, \cref{sec:discussion} discusses some of the choices made in this article and potential ideas for further improvements of the algorithm. 

\section{Background and definition of the proposed numerical method}\label{background}
\subsection{Shape space and shape metrics}
\subsubsection*{Orientation-Preserving Diffeomorphisms}
Consider the space of immersions from $\Omega =[0, 1]$ into euclidean space $\mathbb{R}^d$,   
\begin{equation}
    \Eff := \text{Imm}(\Omega, \mathbb{R}^d) = \left\{c\in C^\infty(\Omega, \mathbb{R}^d) \colon  c'(t)  \neq 0, \; \forall t\in \Omega\right\}.
\label{eq: curve space}
\end{equation}
To be able to identify two curves representing the same shape, consider the set of orientation-preserving diffeomorphisms, which consists of monotonically increasing functions from $\Omega$ onto itself,
\begin{equation}
     \text{Diff}^+(\Omega) = \left\{\varphi \in C^\infty(\Omega, \Omega) \,\colon\, \varphi(0) = 0,\, \varphi(1) = 1,\;  \;  \varphi' > 0\right\}.
\label{eq: reparametrization group}
\end{equation}
The group $\text{Diff}^+(\Omega)$ is an infinite-dimensional Lie group. \Cref{sect:Diffeos} provides more details on this group. For now, we note that $\text{Diff}^+(\Omega)$ has a canonical right group action on $\Eff$ by composition,
\[
    \Eff \times \text{Diff}^+(\Omega) \to \Eff , \qquad (c, \varphi) \mapsto c \circ \varphi.
\]
We shall refer to this as reparametrization of $c$ by $\varphi$. With respect to the reparametrization action, we can now describe shape spaces.

\subsubsection*{Shape Space}
Elements of shape space are interpreted as \textit{unparametrized curves}. The shape space is the set of $\text{Diff}^+(\Omega)$-orbits
\[
 \Ess = \Eff \;/\; \text{Diff}^+(\Omega).
\]
 Our goal is to obtain geometrically sound deformations between elements in $ \Ess $ and compute distances thereof, by means of a metric $d_\Ess$. Such a metric is typically constructed by defining a reparametrization-invariant metric $ d_\Eff $ on $ \Eff $. Due to reparametrization invariance (cf.\ \cite{celledoni2017}), 
 \begin{equation}
    d_\Ess([c_1], [c_2]) = \inf_{\varphi\in\text{Diff}^+(\Omega)}d_\Eff(c_1, c_2\circ\varphi)
\label{eq: Shape Metric}
\end{equation}
defines a metric on \( \Ess \). Hence, to compute the distance between two shapes, we need to solve an optimization problem. The algorithm presented in this article aims to solve this problem. 

\subsubsection*{The SRVT and Q-transform for Curves}
One way to define a suitable distance function $ d_\Eff $ is by exploiting certain transformations which transform familiar metrics to obtain reparametrization- invariant (Riemannian) metrics. One such transformation is  the square-root velocity transform (SRVT) for curves, introduced in \cite{srivastava2011}. It is given by
\[
    \mathcal{R}\colon \Eff  \to C^\infty(\Omega, \mathbb{R}^d) \setminus \{0\}, \quad  c \mapsto  \frac{c'}{\sqrt{\lvert c'\rvert}}.
\]
It is easy to see that $\mathcal{R}(c\circ\varphi)(t) = \sqrt{\varphi'(t)}\; (\mathcal{R}(c)\circ\varphi)(t)$.
Hence we can pull back the $ L^2 $-inner product on function spaces to obtain a reparametrization-invariant distance:
\[ 
    d_\Eff(c_1, c_2) := \left\|\mathcal{R}(c_1) - \mathcal{R}(c_2)\right\|_{L^2(\Omega, \mathbb{R}^d)} = \left(\int_\Omega \lvert\mathcal{R}(c_1)(t) - \mathcal{R}(c_2)(t) \rvert^2\;dt \right)^{1/2}.
\]
The SRVT was generalized to curves on Lie groups and homogeneous spaces \cite{celledoni2016}. 
Another transformation with similar properties as the SRVT, which allows us to construct a different metric, is the
so-called \emph{Q-transform} \cite{mani2010,bauer2014}:
\begin{equation}
    Q\colon  \Eff \to C^\infty(\Omega, \mathbb{R}^d), \qquad c(\cdot) \mapsto \sqrt{\lvert c'(\cdot)\rvert} c(\cdot).
\label{eq: q transform}
\end{equation}

\subsubsection*{Surfaces embedded in 3D space}
To extend the framework for computing shape distances to parametric surfaces embedded in $ \mathbb{R}^3 $, we will assume for simplicity that the surfaces are defined on the unit square $ \Omega = [0, 1]^2 $. The main differences between shapes of curves and of surfaces pertain the group of diffeomorphisms. Further, we need analogues to the SRVT and Q-transform.

Denote by $ f_x, f_y $ the partial derivatives of a parametric surface $ f $ and by
\[
    a_f( \mathbf{x} ) = \lvert f_x( \mathbf{x})  \times f_y( \mathbf{x})  \rvert
\]
the area scaling factor of the surface. Now the pre-shape space for surfaces is 

\begin{equation} 
    \Eff = \left\{ f\in C^\infty(\Omega, \mathbb{R}) \; : \;  a_f(x) > 0  \; \forall x\in\Omega \right\}.
\label{eq: surfaces immersions} 
\end{equation}
 The group of orientation-preserving diffeomorphisms consists of smooth invertible maps from $ \Omega $ onto  itself such that the Jacobian determinant $ J_\varphi $ is positive, 
\[
    \text{Diff}^+(\Omega) = \left\{ \varphi \in C^\infty(\Omega, \Omega) \; : \;  J_\varphi(\mathbf{x}) > 0 \; \forall \mathbf{x}\in \Omega, \; \varphi(s)\in\partial \Omega \; \forall s \in \partial\Omega \right\}. 
\]
Its tangent space $T_\varphi \text{Diff}^+(\Omega)$ at $\varphi$ consists of smooth boundary-respecting vector fields on $ \Omega $ (see~\cite{michor1980}).

\subsubsection*{The SRNF and Q-transform of Surfaces}
The square root normal field (SRNF), defined in \cite{jermyn2012}, is a generalization of the SRVT to surfaces. It is defined by  
\[
    \mathcal{N}(f)(\mathbf{x}) = \sqrt{a_f(\mathbf{x})}\mathbf{n}_f(\mathbf{x)}
\]
where $ \mathbf{n}_f(\mathbf{x}) $ is the normal vector field on $ f $. Applied to a reparametrized surface, the transform satisfies the relation 
\[
    \mathcal{N}(f\circ\varphi) = \sqrt{J_\varphi} (\mathcal{N}(f) \circ \varphi).
\]
Following the same derivation as was given for curves, we may now define a reparametrization-invariant pre-shape distance by  
\[
    d_\Eff(f_1, f_2) = \|\mathcal{N}(f_1) - \mathcal{N}(f_2)\|_{L^2(\Omega, \mathbb{R}^3)} ,
\] 
since the factor $ \sqrt{J_\varphi} $ cancels
the Jacobian determinant introduced by the change of variables in the integral. While the SRNF shares some properties of the SRVT, for example the translation invariance, it is not invertible. In \cite{klassen2019}, the authors present several examples of surfaces that represent different shapes, but have the same image under the SRNF. However, the SRNF framework has shown some merit in practice as shown in \cite{su2020} and the references therein.
 
Alternatively, we can use the $Q$-transform for surfaces, which was introduced in \cite{kurtek2010}. The $Q$-transform for surfaces is defined by
\[ 
    Q(f)(\mathbf{x}) = \sqrt{a_f(\mathbf{x})}f(\mathbf{x})
\]
and satisfies the same property of reparametrization-invariance with respect to the $ L^2 $-norm.

\subsection{The optimization problem on the diffeomorphism group}
The metric on shape space is induced by suitable metrics on the pre-shape space. In the approaches sketched above, this metric is the geodesic distance of a Riemannian metric arising as pullback of the $L^2$-metric via one of the transformations mentioned. To compute the metric on shape space one needs to find the elements in the $\Diff^+ (\Omega)$-orbits which are closest to each other. However, due to equivariance of the transformations, this can equivalently be formulated as the following optimization problem for the $L^2$-distance (exploiting that $(C^\infty (\Omega,\mathbb{R}^d), L^2)$ is a pre-Hilbert space and the set in which we compute the distance is an open (but non-convex) subset of this space). Assume that $f_1,f_2 \in \Eff$, and we want to compute $d_{\Ess} ([f_1],[f_2])$. Let $\mathcal{T}$ be one of the transforms discussed above and $q_i := \mathcal{T}(f_i), i=1,2$. Then we need to solve the following optimization problem:
\begin{align}
\varphi^*={\arg\inf}_{\varphi\in \mathrm{Diff}^+(\Omega)}E(\varphi),\qquad E(\varphi):=\|q_1-\sqrt{\varphi'}\,(q_2\circ \varphi)\|^2_{L^2}.
\label{eq: optimization problem}
\end{align}
Note that the optimization problem has been formulated as the \textit{squared} $L^2$-distance. Solving this optimization problem, we can compute the metric distance on shape space if both $q_1$ and $q_2$ are contained in a convex subset of the image of the transformation. While this is not always the case (for example, it does not hold
for a pair of curves which get mapped by the SRVT to elements $q_1$ and $-q_1$), the assumption always holds locally for mappings which are not too far apart. In any case, solving efficiently this optimization problem is crucial in shape analysis and the main purpose of the present paper is to provide tools from machine learning for its solution.

\subsection{A Riemannian gradient descent method}

A gradient descent approach to solve \eqref{eq: optimization problem} is obtained by representing the gradient of the functional $E$ 
by means of an othonormal basis of $T_{\mathrm{id}}\mathrm{Diff}^+(\Omega)$. Suppose $f_j\in T_{\mathrm{id}}\mathrm{Diff}^+(\Omega)$, $j=1,2\dots$ is an orthonormal basis of $T_{\mathrm{id}}\mathrm{Diff}^+(\Omega)$ with respect to an inner product $\langle \cdot, \cdot \rangle_{\mathrm{id}}$ on $T_{\mathrm{id}}\mathrm{Diff}^+(\Omega)$. 
Consider the left multiplication $L_{\varphi}$ for some $\varphi \in \mathrm{Diff}^+(\Omega)$ mapping $\psi \mapsto \varphi \circ \psi$. 
We denote with $\mathrm{d}L_{\varphi}:T_{\mathrm{id}}\mathrm{Diff}^+(\Omega)\rightarrow T_{\varphi}\mathrm{Diff}^+(\Omega)$ the corresponding derivative mapping at the identity. 
The gradient of $E$ with respect to $\langle \cdot, \cdot \rangle_{\mathrm{id}}$ is defined as the unique
vector field $\mathrm{grad}\, E$ on $\mathrm{Diff}^+(\Omega)$ satisfying
$$
\mathrm{d}E\big\vert_{\varphi}(w_\varphi)=\langle \mathrm{d}L_{\varphi}^{-1} (\mathrm{grad}\, E({\varphi})) , \mathrm{d}L_{\varphi}^{-1} w_{\varphi} \rangle_{\mathrm{id}}, \qquad \forall w_{\varphi}\in T_{\varphi}\mathrm{Diff}^+(\Omega).
$$
We can formally represent $\mathrm{grad}\,  E({\varphi})$ using the basis of $T_{\mathrm{id}}\mathrm{Diff}^+(\Omega)$, by
\begin{equation}
    \label{gradBasis}
\mathrm{d}L_{\varphi}^{-1} (\mathrm{grad}\, E({\varphi}))=\sum_{j=1}^\infty \lambda_j f_j,
\end{equation}
with $\lambda_j=\mathrm{d}E\big\vert_{\varphi}(\mathrm{d}L_{\varphi}(f_j))$.

In practice, the proposed gradient descent algorithm on $\mathrm{Diff}^+(\Omega)$ is obtained by truncating \eqref{gradBasis} and projecting the gradient on the subspace of $T_{\mathrm{id}}\mathrm{Diff}^+(\Omega)$ spanned by a finite number $M$ of basis elements.
This gives
$$\mathrm{d}L_{\varphi}^{-1}\widehat{\mathrm{grad}\, E}(\varphi)=\sum_{j=1}^M \lambda_j f_j,$$
and we obtain the following update rule for the gradient descent iteration:
$$\varphi^{(n+1)}=\varphi^{(n)}\circ \left( \mathrm{id}-\eta\, \mathrm{d}L_{\varphi^{(n)}}^{-1}\widehat{\mathrm{grad}\, E} \big(\varphi^{(n)}\big)\right), \quad n=0,1,\dots,$$
with $\varphi_0=\mathrm{id}$, where $\eta$ is a scalar parameter optimized to speed-up convergence and to guarantee the invertibility of $\varphi^{(n+1)}$.
Written in terms of the basis, the gradient descent iteration becomes
\begin{equation}
    \label{GradDisc}
\varphi^{(n+1)}=\varphi^{(n)}\circ \left( \mathrm{id}-\eta\, \sum_{j=1}^M\lambda_j^n f_j\right), \quad n=0,1,\dots,
\end{equation}
and $\{\lambda_j^n\}_{j=1}^M$ are coefficients determined at each iteration $n$.
In \cite{riseth2021}, this update rule was shown to be equivalent to an algorithm proposed for reparametrization of surfaces in \cite{kurtek2012} and to a similar algorithm created for curves in \cite{srivastava2011}.

\subsection{Main method: Deep learning of diffeomorphisms}\label{sect:deeplearning}
It can be observed that after $L$ iterations of the gradient descent algorithm \eqref{GradDisc}, the diffeomorphism $\varphi^{(L)}$ is the composition of $L$ elementary diffeomorphisms defined as the identity plus a weighted sum of $M$ basis functions.
This observation may be exploited to create a deep learning approach to the reparametrization problem.

The goal is to find a minimizer of the loss function in \eqref{eq: optimization problem}.
However, instead of using the whole group $\Diff^+(\Omega)$, we fix some $L, M \in \N$ and restrict ourselves to functions of the form  
\begin{equation}
    \varphi = \varphi_L\circ \cdots \circ\varphi_1,\quad \mathrm{where}\quad \varphi_{\ell}:=\mathrm{id}+\sum_{n=1}^M w_n^{\ell} f_n,\qquad \ell=1,\dots , L,
    \label{eq:elem_diff}
\end{equation}
for a set of basis function $ \{f_j\}^M_{j=1} \subset T_{\id}\text{Diff}^+(\Omega)$.
Here  $\{w_j^\ell\}_{j=1}^M$ are coefficients to be found by the optimization algorithm.

The computational graph of the function \( \varphi \) in \eqref{eq:elem_diff} has a very similar structure to that of a deep residual network \cite{He}.
Both the input-values \(x_0\) and output-values \(x_L\) of the network are points in the domain \( \Omega \).
The function evaluation is defined by the recursive relation
\[
    \varphi(x_0) = x_L, \qquad x_{\ell+1} = \varphi(x_\ell) = x_\ell + F_{w^\ell}(x_\ell),\qquad \ell=0, ..., L-1.
\]
By using a discrete analogue of \eqref{eq: optimization problem} as the loss function, then the problem of finding an optimal reparametrization between two curves --- i.e. finding the optimal coefficients \( \{w_j^\ell\}_{j=1}^M \) --- coincides with that of training a neural network.
To ensure that $\varphi \in \Diff^+(\Omega)$, we need to restrict the set of weights to some subset $\mathcal{W} \subset \mathbb{R}^{ML}$. Further details on this restriction are given in \cref{sect:methods-curves}.

The increased performance of the deep learning approach \eqref{eq:elem_diff} can be seen in figure \ref{fig: gd comparison}, which compares the proposed method to the gradient descent algorithm. This figure indicates that the deep learning approach converges towards a significantly better solution than gradient descent.
A preliminary implementation of this approach and a proof of concept was presented in \cite{riseth2021}.
{\textwidth=0.8\textwidth
\begin{figure}[htb!]
    \centering
    \includegraphics[width=0.8\textwidth]{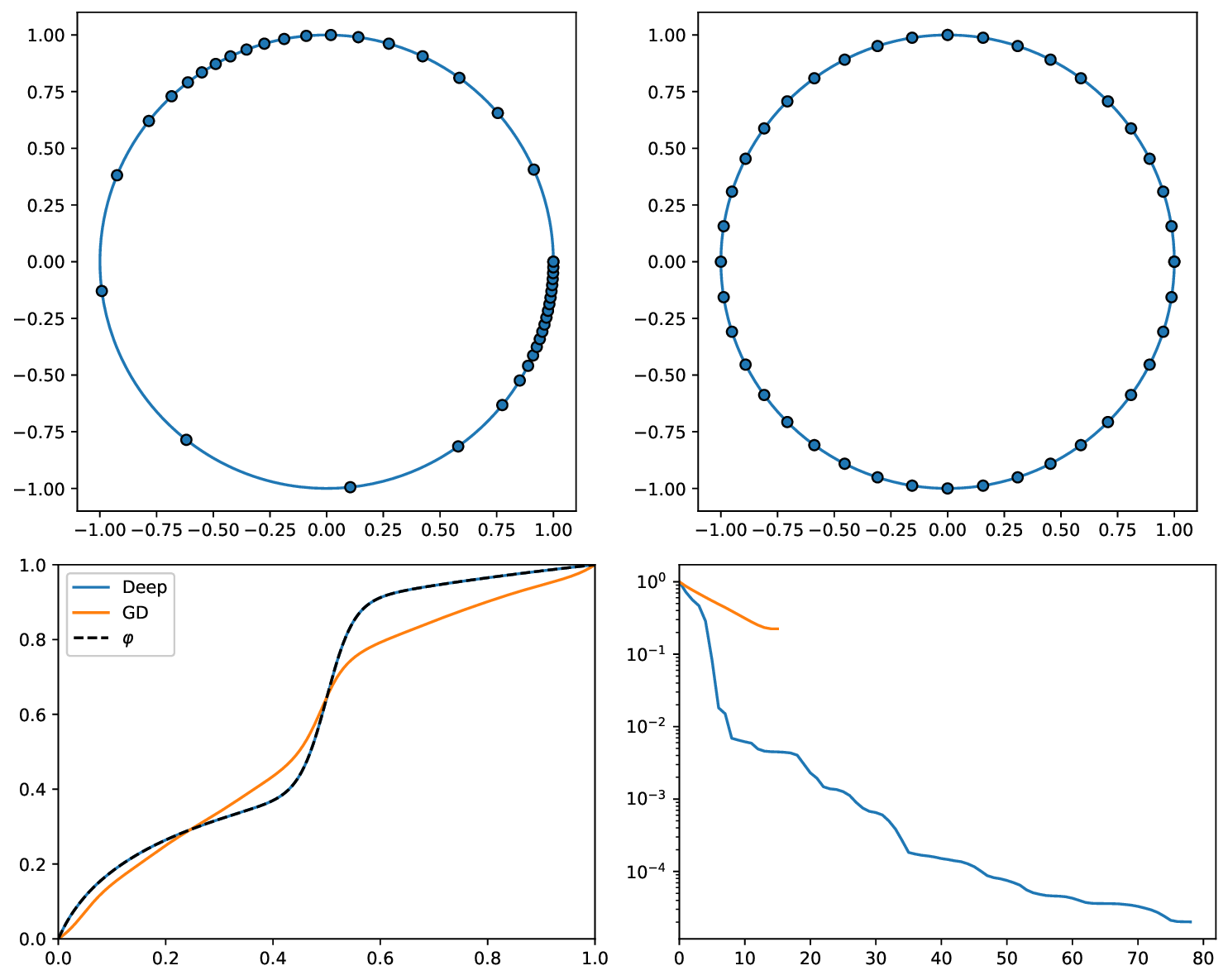}
    \caption{Comparison between the Riemannian gradient descent (GD) and the deep residual network approach for reparametrization of curves. The GD approach used 6 basis functions, whereas the network approach used 6 basis functions and 6 layers. (Top): A curve $c_2$ (left) and the same curve with a different parametrization $c_1=c_2\circ \varphi$ (right). The goal is to reconstruct $\varphi$. (Bottom left): The diffeomorphisms found by the different algorithms, compared to the target diffeomorphism $ \varphi$. The diffeomorphism found by the deep neural network approach is visually indistinguishable from the target. (Bottom right): Convergence plots for the two algorithms, showing the error of the cost function (see \cref{sect:implementation}) as compared to the iteration number of the optimization algorithm. The error is given relative to the initial error.}
    \label{fig: gd comparison}
\end{figure}
}

\section{Theoretical justification}\label{sect:Diffeos}
This section provides the main theoretical motivation for the proposed deep learning approach for optimization on $\mathrm{Diff}^+(\Omega)$.
We shall first recall relevant (infinite-dimensional) Lie group structures on diffeomorphism groups.
Then we show that finite compositions of diffeomorphisms of the type $\varphi_{\ell}=\id+ f_{\ell}$, $\ell=1,\dots , L$ with $f_{\ell}$ a $1$-Lipschitz vector field can be used to describe the whole group of diffeomorphisms of a cube $\Omega=[0,1]^d$ (Section~\ref{sect:cornerresults}).
If $\Omega$ is, more generally, a compact convex subset of ${\mathbb R}^d$ with dense interior (like a disk $\Omega=\{(x,y) \in \mathbb{R}^2 : x^2 + y^2 \leq r^2\}$ in $\mathbb{R}^2$), analogous diffeomorphisms generate the group of diffeomorphisms fixing the boundary of $\Omega$ (see Section~\ref{fixingBoundary}).
We are mostly interested in the case of a cube $[0,1]^d$; in fact, only the unit interval $[0,1]$ and square $[0,1]^2$ are used in the numerical calculations.
As explained in Section~\ref{sect:deeplearning}, eventually our methodology is implemented first restricting $f_{\ell}$ to a finite-dimensional subspace of the Lie algebra of vector fields and then constraining $f_{\ell}$ to be $1$-Lipschitz, see Section~\ref{sec:numres} for details.
This introduces further approximations which we discuss briefly but do not analyze in detail in the present work.

\subsection{Diffeomorphisms fixing the boundary}\label{fixingBoundary}

Lie group structures on diffeomorphism groups of compact manifolds have been studied in many works (see, e.g., \cite{michor1980,hamilton1982,milnor1984,omori1974}), including the case of manifolds with boundary or corners.
In this section, we consider a compact convex subset $\Omega\subseteq {\mathbb R}^d$ with dense interior (having in mind the main example $\Omega=[0,1]^d$).
A map on $\Omega$ is smooth if it extends to a smooth map on an open subset in ${\mathbb R}^d$; we can therefore speak about the group $\Diff(\Omega)$ of all $C^\infty$-diffeomorphisms of~$\Omega$.
Following \cite{gloeckner2017}, we consider the subgroup
\[
    \Diff_\partial (\Omega) := \{\phi \in \Diff (\Omega) : \phi(k) = k, \forall k \in \partial\Omega\}
\]
of all smooth diffeomorphisms fixing the topological boundary $\partial \Omega$ pointwise.  It has been shown in \cite{gloeckner2017} that if we endow $\Diff_\partial (\Omega)$ with the compact-open $C^\infty$-topology, the group becomes a manifold modeled on the space
$$
    C^\infty_\partial (\Omega, \mathbb{R}^d) = \{f \in C^\infty (\Omega,\mathbb{R}^d) : f(k) = 0, \forall k \in \partial \Omega\}.
$$
This space is an infinite-dimensional locally convex space (with respect to the compact-open $C^\infty$-topology) and its elements can be identified with vector fields which vanish on the boundary. This structure turns the space $\Diff_\partial (\Omega)$ into an infinite-dimensional Lie group.
Using a canonical identification yields $T_{\id} \Diff_\partial (\Omega) = C^\infty_\partial (\Omega,\mathbb{R}^d)$. Further, $\Diff_\partial (\Omega)$ admits a global chart  
\begin{align}\label{glob:chart}
    \kappa \colon \Diff_\partial (\Omega) \rightarrow C^\infty_\partial (\Omega,\mathbb{R}^d) ,\quad \phi \mapsto \phi -\id_\Omega,
\end{align}
which arises as a restriction of the map $\kappa \colon C^\infty (\Omega,\mathbb{R}^d) \rightarrow C^\infty (\Omega,\mathbb{R}^d), f\mapsto f - \id_\Omega$.
For vector fields $f\in \kappa(\Diff_\partial (\Omega))$, the inverse $\kappa^{-1}(f) = f + \id_\Omega$ is in $\Diff_\partial (\Omega)$ and we can obtain all elements in $\Diff_\partial (\Omega)$ in this way.

By the preceding, every diffeomorphism fixing the boundary can be expressed as a vector field (vanishing on the boundary) plus the identity. Thus, one only needs to approximate vector fields and can ignore the non-linearity of the infinite-dimensional manifold. However, there are two practical problems for us to solve:
\begin{enumerate}
    \item Describe the image of the global chart $\kappa$.
    \item Establish approximation results by sampling from a finite-dimensional subspace of the Lie algebra.
\end{enumerate}
As the global chart is highly dependent on the geometry of $\Omega$, there is no easy comprehensive way to describe its image. However, it is easy to provide sufficient conditions for a vector field to  lie inside the image of $\kappa$. 

\begin{example}
Let $\Omega =[a,b]$ be a compact interval. Then a vector field $f$ in $C^\infty_\partial ([a,b],\mathbb{R})$ is in the image of $\kappa$ if $f'(x) >-1$ for all $x \in [a,b]$ since then $\id_{[a,b]}+f$ will be monotonically increasing.
\end{example}

The situation is more complicated for higher-dimensional sets. We do not have any global information
concerning the set~$\Omega$ then.
However, there is a convenient sufficient
condition ensuring that a map is in~$\Omega$,
which we describe now.
%

Fix any norm on $\mathbb{R}^d$ to calculate Lipschitz constants for functions from $\Omega\subseteq \mathbb{R}^d$ to~$\mathbb{R}^d$, operator norms,
and open balls $B_r(x):=\{y\in\mathbb{R}^d\colon \|y-x\|<r\}$
for $x\in\mathbb{R}^d$ and $r>0$.
Setting
\begin{equation}
    \label{the-norms}
\|f\|_{\infty,\op}:=\max_{x\in \Omega}\|f'(x)\|_{\op}
\end{equation}
for $f\in C^\infty_{\partial}(\Omega,\mathbb{R}^d)$, we obtain a continuous seminorm on $C^\infty_{\partial}(\Omega,\mathbb{R}^d)$.
Hence
\[
\mathcal{U}_1 :=\{f\in C^\infty_{\partial}(\Omega,\mathbb{R}^d) \colon \Lip(f)<1\}= \{f\in C^\infty_{\partial}(\Omega,\mathbb{R}^d)\colon \|f\|_{\infty,\op}<1\}
\]
is an open $0$-neighborhood in $C^\infty_{\partial}(\Omega,\mathbb{R}^d)$.

\begin{la}\label{la:inchartdomain}
The map $\id_\Omega+f$ is an element of $\Diff_{\partial}(\Omega)$ for all $f\in \mathcal{U}_1$. Thus $\mathcal{U}_1 \subseteq \kappa(\Diff_\partial(\Omega))$ and it makes sense to consider
$\kappa^{-1}(f) = \id_\Omega + f$.
\end{la}

\begin{proof} To keep the notation short, we write $\kappa^{-1}(f) := f + \id_\Omega$, though strictly speaking it makes only sense to use the inverse of the chart once the proof has been concluded.
Given $f\in \mathcal{U}_1$, the map $\kappa^{-1}(f) \colon \Omega\to\mathbb{R}^d$ is injective, since $\kappa^{-1}(f)(x)=\kappa^{-1}(f)(y)$ with $x\not=y$ in~$\Omega$ would entail that
\[
\|x-y\|=\|f(y)-f(x)\|\leq \Lip(f)\|x-y\|<\|x-y\|,
\]
contradiction.
Since $\kappa^{-1}(f)(x)=x$ for each $x\in \partial \Omega$, we have
\begin{equation}\label{eq_boundarypreservation}
    \kappa^{-1}(f)(\partial \Omega)=\partial \Omega.
\end{equation}
Using the injectivity of $\kappa^{-1}(f)$, we obtain for the interior $\Omega^\circ$ of $\Omega$
\begin{equation}\label{inside-outside}
    \kappa^{-1}(f)(\Omega^\circ)\cap\partial \Omega=\emptyset.
\end{equation}
Since $\|f'(x)\|_{\op}<1$, we have $\kappa^{-1}(f)'(x)=\id_{\mathbb{R}^d}+f'(x)\in\GL(\mathbb{R}^d)$ for each $x\in \Omega^\circ$, whence $\kappa^{-1}(f)\vert_{\Omega^\circ}$ is a local $C^\infty$-diffeomorphism at each $x\in \Omega^\circ$ (by the Inverse Function Theorem) and hence an open map.

Next, we pick an $x_0$ in the interior $\Omega^\circ$ of~$\Omega$.  Then the distance
\[
    r:=\min\{\|x_0-y\|\colon y\in \partial \Omega\}
\]
is positive.  Note that
\begin{equation}\label{ball-inside}
B_r(x_0):=\{y\in \mathbb{R}^d\colon \|x_0-y\|<r\}\subseteq \Omega^\circ .
\end{equation}
In fact, since $B_r(x_0)\cap\partial \Omega=\emptyset$, the disjoint open sets $B_r(x_0)\cap \Omega^\circ
\not=\emptyset$ and $Q:=B_r(x_0)\cap (\mathbb{R}^d\setminus \Omega)$ cover~$B_r(x_0)$.  Since $B_r(x_0)$ is connected, $Q=\emptyset$ follows.\\[2mm]
There exists $y_0\in\partial \Omega$ such that $\|x_0-y_0\|=r$. Since $f(y_0)=0$, we have
\[
\|f(x_0)\|=\|f(x_0)-f(y_0)\|\leq \Lip(f)\|x_0-y_0\|<r.
\]
Hence
\[
\|\kappa^{-1}(f)(x_0)-x_0\|=\|f(x_0)\|<r,
\]
whence $\kappa^{-1}(f)(x_0)\in \Omega^\circ$, by (\ref{ball-inside}).  Since $x_0\in \Omega^\circ$ was arbitrary,
\[
    \kappa^{-1}(f)(\Omega^\circ)\subseteq \Omega^\circ
\]
is established.  Now $\kappa^{-1}(f)(\Omega)$ is compact and hence closed in $\mathbb{R}^d$. Furthermore, since ${\kappa^{-1}(f)(\partial \Omega)=\partial \Omega}$, we deduce that
\[
\kappa^{-1}(f)(\Omega^\circ)=\kappa^{-1}(f)(\Omega^\circ)\cap \Omega^\circ=\kappa^{-1}(f)(\Omega)\cap \Omega^\circ
\]
is closed in~$\Omega^\circ$.  Since $\Omega^\circ$ is connected and its non-empty subset $\kappa^{-1}(f)(\Omega^\circ)$ is both open and closed in $\Omega^\circ$, we must have $\kappa^{-1}(f)(\Omega^\circ)=\Omega^\circ$.  Being a bijective local $C^\infty$-diffeomorphism, the map $\kappa^{-1}(f)\vert_{\Omega^\circ}\colon \Omega^\circ\to \Omega^\circ$ is a $C^\infty$-diffeomorphism.  Since $f\vert_{\partial \Omega}=0$, we have $\kappa^{-1}(f)\vert_{\partial \Omega}=\id_{\partial \Omega}$, whence $\kappa^{-1}(f)(\Omega)=\Omega$.  It only remains to show that $\kappa^{-1}(f)\colon \Omega\to \Omega$ is a $C^\infty$-diffeomorphism.  By the preceding, $\kappa^{-1}(f)\colon \Omega\to \Omega$ is a continuous bijective self-map of the compact topological space~$\Omega$ an hence a homeomorphism.  Note that $\kappa^{-1}(f)'(x)=\id_{\mathbb{R}^d}+f'(x)\in \GL(\mathbb{R}^d)$ for all $x\in \Omega$.  If $k\in \N_0$ and $\kappa^{-1}(f)$ is a $C^k$-diffeomorphism, then
\[
(\kappa^{-1}(f)^{-1})'(x)=\big(\kappa^{-1}(f)'(\kappa^{-1}(f)^{-1}(x))\big)^{-1}
\]
is a $\GL(\mathbb{R}^d)$-valued $C^k$-function of $x\in \Omega$, entailing that $\kappa^{-1}(f)^{-1}$ is $C^{k+1}$.  Thus $\kappa^{-1}(f)^{-1}$ is smooth and thus $\kappa^{-1}(f)\in\Diff_{\partial}(\Omega)$.
\end{proof}

\begin{rem}
    The same reasoning may be used to show that, for each $\ell\in \N$, the set ${\{f\in C^\ell_{\partial}(\Omega ,\mathbb{R}^d) \colon \Lip(f)<1\}}$ is an open $0$-neighborhood in $C^\ell_{\partial}(\Omega ,\mathbb{R}^d)$ and $\id_\Omega+ f\in \Diff_{\partial}^{C^\ell}(\Omega)$ for all $f$ in this $0$-neighborhood.
\end{rem}

The main disadvantage of the group $\Diff_\partial (\Omega)$ in the applications we wish to investigate is that its elements fix the boundary pointwise. Depending on the application, this can be quite unnatural. Hence, we also consider the full group of $C^\infty$-diffeomorphisms for the special case of a cube $[0,1]^d \subseteq \mathbb{R}^d$.

\subsection{Diffeomorphisms of a cube}\label{sect:cornerresults}

In this section, we exclusively consider the cube $\Omega:=[0,1]^d$. Then the group $\Diff(\Omega)$ of smooth diffeomorphisms is an infinite-dimensional Lie group (cf.\ \cite{michor1980}).  We shall follow the complementary approach of~\cite{gloeckner2022}, where diffeomorphism groups of convex polytopes are discussed. Our goal is to generalize Lemma \ref{la:inchartdomain} to the  diffeomorphism group of a cube.

For the cube, the diffeomorphism group $\Diff (\Omega)$ is an infinite-dimensional Lie group modeled on the Fr\'{e}chet space $C^\infty_{\text{str}} (\Omega,\mathbb{R}^d)\subseteq C^\infty(\Omega,\mathbb{R}^d)$ of smooth vector fields
$$
f=(f_1,\ldots, f_n)\colon \Omega\to{\mathbb R}^d
$$
which are \emph{tangent to the boundary} in the sense that $f_j(x_1,\ldots,x_d)=0$ whenever $x_j=0$ or $x_j=1$ (such vector fields are also called boundary respecting, or stratified vector fields).  For example, a smooth vector field $f\colon [0,1]\to\mathbb{R}$ is tangent to the boundary if it vanishes at the vertices $0$ and $1$.  A smooth vector field $f=(f_1,f_2)\colon [0,1]^2\to {\mathbb R}^2$ is tangent to the boundary if $f_1(x,y)$ vanishes whenever $x=0$ or $x=1$, and $f_2(x,y)$ vanishes whenever $y=0$ or $y=1$.
The space of stratified vector fields is a closed vector subspace of $C^\infty (\Omega,\mathbb{R}^d)$.  As in every Lie group, the connected component $\Diff(\Omega)_0$ of the neutral element $\id_\Omega$ is an open subgroup of $\Diff(\Omega)$.  We recall from \cite{gloeckner2022}:

\begin{la}
Let $\Omega = [0,1]^d$. Then $\kappa \colon \Diff (\Omega)_0 \rightarrow C^\infty_{\mathrm{str}} (\Omega,\mathbb{R}^d), \phi \mapsto \phi -\id_\Omega$ has open image and is a chart of $\Diff (\Omega)$, when considered as a map onto the image. In particular, $\Diff_\partial (\Omega)$ is a Lie subgroup of $\Diff (\Omega)$.
\end{la}

Again, it is not straightforward to determine the image of the chart $\kappa$.  But we can extend the sufficient criterion encountered for boundary-fixing diffeomorphisms.  To this end, we consider the open subset
$$\mathcal{V}_1 = \{f \in C^\infty_{\text{str}}(\Omega,\mathbb{R}^d) : \text{Lip} (f) < 1\}.$$
We prove the following (see \cite[Proposition 1.5]{gloeckner2022} for a generalized version):
\begin{prop}
For every $f \in \mathcal{V}_1$, the map $f+\id_\Omega$ is contained in $\Diff (\Omega)$. Hence $\mathcal{V}_1 \subseteq \kappa (\Diff(\Omega)_0)$.
\end{prop}

\begin{proof}
The proof is by induction on $d\in\N$.
If $d=1$, then $C^\infty_{\text{str}}(\Omega,\mathbb{R})=C^\infty_\partial(\Omega,\mathbb{R})$ and the assertion follows from Lemma~\ref{la:inchartdomain}.
Now assume that $d\geq 2$.  Fix $f \in \mathcal{V}_1$. Arguing as in the proof of Lemma \ref{la:inchartdomain}, we see that $\id_\Omega +f$ is injective.
Note that each $(d-1)$-dimensional face~$F$ of $\Omega$ is diffeomorphic to $[0,1]^{d-1}$ via a translation; for example, the upper face $[0,1]\times \{1\}$ of the square $[0,1]^2$ can be translated to $[0,1]\times \{0\} \cong [0,1]$.
Hence $(\id_\Omega+f)(F)=(\id_F+f\vert_F)(F)=F$ by induction.
As the boundary $\partial \Omega$ is the union of the $(d-1)$-dimensional faces, we deduce that $(\id_\Omega+f)(\partial\Omega)=\partial \Omega$.
Thus
\eqref{eq_boundarypreservation} holds.
Next, we note that $\id_\Omega+f$ has invertible derivative $f'(x)$ at each $x \in \Omega^\circ$, whence it is an open map on the interior~$\Omega^\circ$.
To see that $(\id_\Omega + f)(\Omega) \subseteq \Omega$, consider the projection $\text{pr}_j$ onto the $j$th component for $j \in \{1,\dots , d\}$.
By compactness of $\Omega$, the minimum $m_j := \min_{y \in \Omega} (\text{pr}\circ (\id_\Omega + f)(y))$ exists.
Since $\id_{\Omega^\circ} +f\vert_{\Omega^\circ}$ is an open map, the minimum cannot be attained on the interior of the cube.
Thus $m_j$ is attained on the boundary~$\partial \Omega$.
As the boundary is mapped to itself and does not contain points with negative coordinates, we conclude that $m_j \geq 0$ for every $j$, whence $x_j+f_j(x)\geq 0$ for each $x=(x_1,\ldots,x_d)\in \Omega$.
Considering the maximum instead, we conclude that also $x_j+f_j(x)\leq 1$, whence $x+f(x)\in\Omega$ for all $x\in\Omega$.
Thus $\id_\Omega + f$ maps the sets $\Omega$ and $\partial \Omega$ into themselves, respectively.
Arguing as in the proof of Lemma \ref{la:inchartdomain}, we see that $(\id_\Omega + f)(\Omega^\circ)$ is both open and closed in the connected set $\Omega^\circ$, whence $(\id_\Omega+f)(\Omega^\circ)=\Omega^\circ$.
Then $(\id_\Omega+f)(\Omega)=(\id_\Omega+f)(\Omega^\circ) \cup(\id_\Omega +f)(\partial \Omega)=\Omega^\circ\cup\partial\Omega=\Omega$, showing that $\id_\Omega+f$ is surjective as a map to $\Omega$ and hence bijective.
The proof can now be completed as in the proof of Lemma \ref{la:inchartdomain}, and we deduce that $\id_\Omega + f \in \Diff (\Omega)$.
\end{proof}

%
We have the following (strict) inclusions concerning the sets we are working with

\begin{tikzcd}
\mathcal{U}_1 
\arrow[r] \arrow[d] &
\mathcal{V}_1 \arrow[d] & \\
 \kappa(\mathrm{Diff}_\partial (\Omega))\arrow[r] &
 \kappa (\mathrm{Diff} (\Omega)_0) \arrow[r] &
 C^\infty (\Omega, \mathbb{R}^d).
\end{tikzcd}

%

In practice, we work with vector fields to approximate diffeomorphisms. As the algebra of vector fields is infinite dimensional, we have to replace it by a finite-dimensional subspace. To increase the supply of diffeomorphisms we can approximate, our method incorporates \emph{compositions} of diffeomorphisms. Since $\kappa^{-1}(\mathcal{U}_1)$ and $\kappa^{-1}(\mathcal{V}_1)$ are identity neighborhoods in the respective diffeomorphism groups, they generate the group,\footnote{We are referring to the connected component of the identity of the respective diffeomorphism group here.}
i.e.\ every diffeomorphism can be written as a finite (but maybe arbitrarily long) composition of elements near the identity. 
Passing to the intersection with a finite-dimensional subspace, the generating property is lost, but compositions will nevertheless yield a much broader spectrum of diffeomorphisms which can be reached.

\begin{rem}
These observations have been exploited in the
\emph{boundaryless} case for example in \cite{agrachev2009}. Since $\Diff (\Omega)_0$ is simple in this case, there are several characterizations available for a subset of vector fields to generate a sufficiently rich subgroup (see e.g.\ \cite{agrachev2009,agrachev2021} for a treatment of flows in the framework of sub-Riemannian geometry on manifolds \emph{without} boundary). Note that it is not clear how these techniques (e.g.\ sub-Riemannian geometry) can be generalized to manifolds with boundary.
\end{rem} 
Since our method relies on compositions of multiple diffeomorphisms, the next section provides insight and estimates into such compositions.

\section{Estimates concerning multiple compositions}\label{sect:multcomp}

If we consider compositions $\phi:=\phi_L\circ\cdots\circ \phi_1$ of diffeomorphisms $\phi_1,\ldots,\phi_L\in \Diff_{\partial}(\Omega)$ for a large number~$L$ of factors, we'd like to be able to control the distance $\|\phi-\id_\Omega\|_{C^k}$ of the composite from the identity $\id_\Omega$ in terms of the distances $\|\phi_1-\id_\Omega\|_{C^k},\ldots,\|\phi_L-\id_\Omega\|_{C^k}$ (where the norms $\|\cdot\|_{C^k}$ are as in \eqref{the-norms-ck}).  We show that such an estimate is possible using a closed formula which applies to all~$L$. 
Let $\Omega$ be $[0,1]^d$, or a compact convex subset of $\mathbb{R}^d$ with dense interior.
For certain positive integers $M_0\leq M_1\leq M_2\leq\cdots$ which can be calculated recursively, we have:
\begin{prop}\label{est-iter}
For all $k\in \N_0$ and all $L\in \N$,
\begin{eqnarray}
\lefteqn{\|((\id_\Omega+f_L)\circ \cdots \circ (\id_\Omega+f_1))-\id_\Omega\|_{C^k}}\qquad\qquad
\notag \\[.7mm]
&\leq & M_k\, e^{k\sum_{j=1}^L\|f_j\|_{C^k}}\,
(\|f_1\|_{C^k}+\cdots+\|f_L\|_{C^k})\label{est-it-1}
\end{eqnarray}
holds for all $f_1,\ldots, f_L\in \Diff_{\partial}(\Omega)-\id_\Omega$
such that $\|f_1\|_{C^k}+\cdots+\|f_L\|_{C^k}\leq 1$.
Notably,
\begin{equation}\label{est-it-2}
\|((\id_\Omega+f_L)\circ \cdots \circ (\id_\Omega+f_1))-\id_\Omega\|_{C^k}
\leq M_k\, e^k (\|f_1\|_{C^k}+\cdots+\|f_L\|_{C^k}).
\end{equation}
\end{prop}

Here $M_0:=M_1:=1$. For $k\in \N$ and $j\in \{1,\ldots, k\}$, let $P_{k,j}$ be the set
of all partitions $P=\{I_1,\ldots, I_j\}$ of $\{1,\ldots, k\}$ into disjoint,
non-empty subsets $I_1,\ldots, I_j$ of $\{1,\ldots, k\}$.
Then
\begin{equation}\label{the-C-k}
M_k:=\sum_{j=2}^k\sum_{P\in P_{k,j}}M_{\lvert I_1 \rvert}\cdots M_{\lvert I_j \rvert}
\;\; \mbox{for $\,k\geq 2$.}
\end{equation}
If $n_1\geq n_2\geq\cdots\geq n_j\geq 1$ are integers with $n_1+\cdots+n_j=k$, write
$P(n_1,\ldots, n_j)\subseteq P_j$ for the set of all partitions of $\{1,\ldots, k\}$
into subsets $I_1,\ldots, I_j$ of cardinality $\lvert I_a \rvert=n_a$ for $a\in\{1,\ldots, j\}$.
Then
\[
M_k\, =\, \sum_{j=2}^k\, \sum_{n_1+\cdots+n_j=k}
\lvert
P(n_1,\ldots, n_j)
\rvert
\, M_{n_1}\cdots M_{n_j},
\]
using a summation over all $(n_1,\ldots,n_j)\in\N^j$ with $n_1\geq n_2\geq\cdots\geq n_j$ and $n_1+\cdots+n_j=k$.\\[2.3mm]
For small $k$, the numbers $M_k$ can easily be calculated by hand; their values are as follows:\\[3.8mm]
\hspace*{-1.5mm}
\resizebox{\textwidth}{!} {%
\begin{tabular}{||c||c|c|c|c|c|c|c|c|c|c||}\hline\hline
$k$ & 1 &2 &3 &4 &5 &6 &7 &8 &9 & 10\\ \hline
$M_k$ & 1 & 1 & 4 & 26 & 236 &
2752 & 39208 & 660032 & 12818912 & 282137824 \\ \hline\hline
\end{tabular}
}\\[3.8mm]
The constants may not be optimal. Note that, for every $k\geq 3$, the set $\{1,\ldots,k\}$ admits $k$ partitions into a singleton and
a set with $k-1$ elements. Thus $M_k\geq k M_1M_{k-1}\geq M_{k-1}$.
\begin{rem}
The numbers $M_k$ for $k\in\N$ are well known in combinatorics; they arise in a classical enumeration problem known as Schröder's Fourth
Problem (cf.\ \cite{riordan1976,schroeder1870}). A longer list can be found in the on-line encyclopedia of integer sequences (OEIS), sequence A000311.
\end{rem}

\begin{rem}\label{firem}
For $m\in \N$, the assertions of Proposition~\ref{est-iter} remain valid for all $k\in\{0,1,\ldots,m\}$
if $f_1,\ldots, f_L\in \Diff^{C^m}_{\partial}(\Omega)-\id_\Omega\subseteq C^m_{\partial}(\Omega ,\mathbb{R}^d)$
(the proof carries over).
\end{rem}
In the following,
$\mathcal{V} := \kappa(\Diff_\partial (\Omega)) = \Diff_{\partial}(\Omega)-\id_\Omega$,
which is an open $0$-neighborhood in $C^\infty_{\partial}(\Omega ,\mathbb{R}^d)$.
\begin{rem}\label{secrem}
The proof will show that the condition $\|f_1\|_{C^k}+\cdots+\|f_L\|_{C^k}\leq 1$
in Proposition~\ref{est-iter} is unnecessary for the validity of~(\ref{est-it-1})
in the cases $k=0$ and $k=1$. Actually,
\begin{equation}\label{onlycts}
\|((\id_\Omega+f_L)\circ \cdots \circ (\id_\Omega+f_1))-\id_\Omega\|_\infty\leq
\|f_1\|_\infty+\cdots+\|f_L\|_\infty
\end{equation}
for all $L\in \N$ and $f_1,\ldots, f_L\in \mathcal{V}$. We shall also see that
\begin{alignat}{8}
    \Lip(((\id_\Omega+f_L)\circ & \cdots\circ (\id_\Omega+f_1))-\id_\Omega)\\
    \leq& (1+\Lip(f_1))\cdots(1+\Lip(f_L))-1 \label{onlylip1}\\
    \leq& e^{\Lip(f_1)+\cdots+\Lip(f_L)}-1 \label{onlylip2} \\
    \leq& e^{\Lip(f_1)+\cdots+\Lip(f_L)}(\Lip(f_1)+\cdots+\Lip(f_L))\label{onlylip3}
\end{alignat}
%
for all $L\in \N$ and $f_1,\ldots, f_L\in  \mathcal{V}$.
\end{rem}
\begin{rem}
In the special case $\Omega=[0,1]^d$, the conclusion of Proposition~\ref{est-iter} remains valid if $\Diff_\partial(\Omega)$ is replaced with $\Diff(\Omega)_0$ and vector fields vanishing on the boundary are replaced with vector fields which are tangent to the boundary.  The same comment applies to Remarks~\ref{firem} and \ref{secrem}.
\end{rem}

The following notations and facts are useful for the proof of Proposition~\ref{est-iter}.
\begin{numba}
If $(E,\|\cdot\|_E)$ is a normed space, we write $\wb{B}_r^E(x):=\{y\in E\colon \|y-x\|_E\leq r\}$ for the closed ball of radius $r>0$ around $x\in E$. If also $(F,\|\cdot\|_F)$ is a normed space, $k\in\N$ and $\beta\colon E^k\to F$ a continuous $k$-linear map, we define
\[
\|\beta\|_{\op}:=\sup\{\|\beta(x_1,\ldots, x_k)\|_F\colon x_1,\ldots,x_k\in\wb{B}_1^E(0)\}.
\]
Then $\|\beta(x_1,\ldots,x_k)\|_F\leq\|\beta\|_{\op}\|x_1\|_E\cdots\|x_k\|_E$ for all $x_1,\ldots, x_k\in E$.  If $\Omega$ is a compact topological space and $g\colon \Omega\to (\cL^k(E,F),\|\cdot\|_{\op})$ a continuous map to the space of continuous $k$-linear maps, we let
\[
\|g\|_{\infty,\op}:=\sup_{x\in \Omega}\|g(x)\|_{\op}\in [0,\infty[\, .
\]
\end{numba}
\begin{numba}
Let $E:=\mathbb{R}^d$, endowed with a fixed norm $\|\cdot\|$, and $\Omega \subseteq E$ be a compact convex subset with non-empty interior.  For $f\in C^\infty(\Omega,E)$
and $k\in\N_0$, let $d^kf\colon \Omega\times E^k\to E$
be the continuous map such that, for all
$x\in \Omega^\circ$ and $y_1,\ldots, y_k\in E$,
\[
d^kf(x,y_1,\ldots, y_k):=(D_{y_k}\cdots D_{y_1}f)(x)
\]
is the iterated directional derivative at $x$
in the directions $y_1,\ldots, y_k$.
Recall that the topology on $C^\infty_{\partial}(\Omega ,E)$ is initial with respect to the mappings $C^\infty_{\partial}(\Omega ,E)\to C(\Omega\times E^k,E)$, $f\mapsto d^kf$ for $k\in\N_0$, using the topology of compact convergence on spaces of continuous functions.  Thus, the seminorms
\[
f\mapsto \sup\{\|d^kf(z)\|_E\colon z\in K\}
\]
define the locally convex vector topology on $C^\infty_{\partial}(\Omega ,E)$, for $k\in \N_0$ and compact subsets $K\subseteq \Omega \times E^k$.  Abbreviate $f^{(k)}(x):=d^kf(x,\cdot)\in \cL^k(E,E)$ for $k\in \N$ and $x\in \Omega$. Then
\[
f\mapsto \|f^{(k)}\|_{\infty,\op}=\sup\{\|d^kf(z)\|_E \colon z\in \Omega\times\wb{B}^E_1(0)^k\}
\]
is a continuous seminorm on $C^\infty_{\partial}(\Omega ,E)$. Together with the supremum norm $\|\cdot\|_\infty$, the latter seminorms define the locally convex topological
vector space topology of $C^\infty_{\partial}(\Omega ,E)$.  In fact, for each $k\in \N$ and compact subset $K\subseteq \Omega\times E^k$, we have $K\subseteq \Omega\times \wb{B}^E_r(0)^k$ for some $r>0$, whence
\[
\sup\{\|d^kf(x,z)\|_E\colon z\in K\}\leq r^k \|f^{(k)}\|_{\infty,\op}.
\]
Hence, the vector topology on $C^\infty_{\partial}(\Omega ,E)$ is also defined by the sequence of norms $(\|\cdot\|_{C^k})_{k\in\N_0}$ with $\|\cdot\|_{C^0}:=\|\cdot\|_\infty$ and 
\begin{equation}\label{the-norms-ck}
\|f\|_{C^k}:=\max\{\|f\|_\infty,\|f^{(1)}\|_{\infty,\op},\ldots, \|f^{(k)}\|_{\infty,\op}\}.
\end{equation}
\end{numba}
The following formula actually holds
for locally convex spaces~$X$, $Y$, and~$Z$.
\begin{numba}(Fa\'{a} di Bruno's Formula)
Let $X$, $Y$, and $Z$ be finite-dimensional real vector spaces, $U\subseteq X$ and $V\subseteq Y$ be locally convex subsets with dense interior, and $k\in \N$.  Let $g\colon U\to Y$ and $f\colon V\to Z$ be $C^k$-maps such that $g(U)\subseteq V$.  Then
\[
d^k(f\circ g)(x,y_1,\ldots, y_k)
=\sum_{j=1}^k
\sum_{P\in P_{k,j}}
d^jf(d^{\lvert I_1 \rvert}g(x,y_{I_1}),\ldots,d^{\lvert I_j \rvert}g(x,y_{I_j}))
\]
for all $x\in U$ and $y_1,\ldots, y_k\in E$, where $P=\{I_1,\ldots,I_j\}$ and
\[
y_J:=(y_{i_1},\ldots, y_{i_\ell})
\]
for each subset $J\subseteq \{1,\ldots,k\}$ of the form $J=\{i_1,\ldots,i_\ell\}$ with $i_1<\cdots<i_\ell$ (see \cite[Theorem 1.3.18 and Remark 1.4.15\,(c)]{gloeckner2015}; cf.\ also \cite{clark2012}).
\end{numba}
\begin{numba}
We shall use the following fact: If $L\in \N$ and $x_1,\ldots, x_L\geq 0$, then
\begin{equation}\label{exp1}
(1+x_1)(1+x_2)\cdots(1+x_L)\leq e^{x_1+\cdots+x_L}
\end{equation}
as $\ln((1+x_1)\cdots (1+x_L))=\sum_{j=1}^L\ln(1+x_j)\leq\sum_{j=1}^L x_j$.  We shall also use that
\begin{equation}\label{exp2}
e^x-1\leq e^x \,x\;\;\mbox{for all $x\geq 0$.}
\end{equation}
In fact, $e^x-1=\exp(x)-\exp(0)=\exp'(\xi)x$ for some $\xi\in [0,x]$ by the Mean Value Theorem, with $\exp'(\xi)=\exp(\xi)\leq \exp(x)$ by monotonicity of $\exp$.
\end{numba}
\begin{numba}
Let $(E,\|\cdot\|)$ be a normed space and $(\cL(E),\|\cdot\|_{\op})$ be the corresponding algebra of bounded operators. For all $L\in \N$ and $\alpha_1,\ldots,\alpha_L\in\cL(E)$, we trivially have
\begin{equation}\label{iterop-1}
\|(\id_E+\alpha_L)\circ\cdots\circ(\id_E+\alpha_1)\|_{\op}
\leq (1+\|\alpha_L\|_{\op})\cdots(1+\|\alpha_1\|_{\op}).
\end{equation}
Moreover,
\begin{eqnarray}
\lefteqn{
\|((\id_E+\alpha_L)\circ \cdots\circ (\id_E+\alpha_1))-\id_E\|_{\op}}\qquad\qquad \notag \\
& \leq & ((1+\|\alpha_L\|_{\op})\cdots(1+\|\alpha_1\|_{\op}))-1 \label{iterop-2}\\
&\leq & e^{\|\alpha_1\|_{\op}+\cdots+\|\alpha_L\|_{\op}}-1\label{iterop-3}\\
&\leq & e^{\|\alpha_1\|_{\op}+\cdots+\|\alpha_L\|_{\op}}
(\|\alpha_1\|_{\op}+\cdots+\|\alpha_L\|_{\op}).\label{iterop-4}
\end{eqnarray}
In fact,
\[
\beta:=(\id_E+\alpha_1)\cdots(\id_E+ \alpha_L)-\id_E
=\sum_{j=1}^m\sum_{\lvert I \rvert=j}\alpha_{i_1}\circ\cdots\circ \alpha_{i_j},
\]
using subsets $I=\{i_1,\ldots, i_j\}\subseteq \{1,\ldots,L\}$ with $i_1<\cdots < i_j$. Thus
\[
\|\beta\|_{\op}\leq
\sum_{j=1}^L\sum_{\lvert I \rvert=j}\|\alpha_{i_1}\|_{\op}\cdots \|\alpha_{i_j}\|_{\op}
=((1+\|\alpha_1\|_{\op})\cdots(1+\|\alpha_L\|_{\op}))-1,
\]
establishing (\ref{iterop-2}). Using~(\ref{exp1}) and~(\ref{exp2}), the inequalities~(\ref{iterop-3}) and~(\ref{iterop-4}) follow.
\end{numba}

\begin{proof}[Proof of Proposition~\ref{est-iter}]
We first prove the assertion for $k\in\{0,1\}$.\\[2.3mm]
\emph{The case $k=0$}. For $f,g\in \mathcal{V}$, let
\[
f*g:=((\id_\Omega+f)\circ (\id_\Omega+g))-\id_\Omega=g+f\circ (\id_\Omega+ g).
\]
Then $f*g\in\mathcal{V}$ and
\[
\|f*g\|_\infty\leq \|g\|_\infty+\|f\|_\infty.
\]
Given $f_1,\ldots, f_L\in \mathcal{V}$,
we define $f_L*\cdots*f_1$ recursively as
$f_L*(f_{L-1}*\cdots*f_1)$ and obtain
\begin{equation}\label{the-k-1}
\|f_L*\cdots*f_1\|_\infty\leq \|f_1\|_\infty+\cdots+\|f_L\|_\infty,
\end{equation}
by a straightforward induction. By induction on $2\leq L\in\N$, also
\[
((\id_\Omega+f_L)\circ\dots\circ (\id_\Omega+f_1))-\id_\Omega=f_L*\cdots*f_1,
\]
as the left-hand side of the equation equals
\begin{eqnarray*}
\lefteqn{
(\id_\Omega+f_L)(\id_\Omega+((\id_\Omega+f_{L-1})\circ\cdots\circ (\id_\Omega+f_1)-\id_\Omega))-\id_\Omega}
\qquad\qquad \\
& =& f_L*(((\id_\Omega+f_{L-1})\circ\cdots\circ (\id_\Omega+f_1))-\id_\Omega)\\
&=& f_L*(f_{L-1}*\cdots*f_1).
\end{eqnarray*}
Hence (\ref{onlycts}) is a re-writing of~(\ref{the-k-1}). Notably, (\ref{est-it-1}) holds for $k=0$ with $M_0:=1$.\\[2.3mm]
\emph{The case $k=1$}. Abbreviate $E:=\mathbb{R}^d$. For $f_1,\ldots, f_L\in \mathcal{V}$, $x_0\in \Omega$ and $x_{j}:=f_{j}(x_{j-1})$ for
$j\in\{1, \ldots, L-1\}$, we have
\[
\begin{split}
    (((\id_\Omega + f_L)&\circ\cdots\circ (\id_\Omega+f_1))-\id_\Omega)'(x_0) \\
                        &= ((\id_E+f_1'(x_0))\circ\cdots\circ (\id_E+f_L'(x_{L-1})))-\id_E
\end{split}
\]
with operator norm
$\leq ((1+\|f_1'(x_0)\|_{\op})\cdots(1+\|f_L'(x_{L-1})\|_{\op}))-1$, by~(\ref{iterop-2}). Passing to the supremum in $x_0\in \Omega$, we get
\[
\|(f_L*\cdots* f_1)'\|_{\infty,\op}\leq ((1+\|f_1'\|_{\infty,\op})\cdots (1+\|f_L'\|_{\infty,\op}))-1,
\]
which is bounded by $((1+\|f_1\|_{C^1})\cdots(1+\|f_L\|_{C^1}))-1$. Thus (\ref{onlylip1}) holds, and (\ref{onlylip2}) as well as (\ref{onlylip3})
follow using~(\ref{exp1}) and~(\ref{exp2}). As the latter number is also an upper bound for
$\|f_L*\cdots*f_1\|_\infty$, we deduce that
\[
\|f_L*\cdots *f_1\|_{C^1}\leq ((1+\|f_1\|_{C^1})\cdots(1+\|f_L\|_{C^1}))-1
\]
for all $L\in\N$ and $f_1,\ldots,f_L\in \Omega$.
By~(\ref{exp1}) and~(\ref{exp2}), the right-hand side is bounded by
\[
e^{\|f_1\|_{C^1}+\cdots+\|f_L\|_{C^1}}-1\leq
e^{\|f_1\|_{C^1}+\cdots+\|f_L\|_{C^1}}(\|f_1\|_{C^1}+\cdots+\|f_L\|_{C^1}),
\]
for all $L\in\N$ and $f_1,\ldots, f_L\in\mathcal{V}$.
Notably, (\ref{est-it-1}) holds for $k=1$ with $M_1:=1$.\\[2.3mm]
\noindent
For $k \geq 2$, the proof is by induction. Details are postponed to Appendix \ref{app:details}.
\end{proof}
\vspace{2mm}
The upper bounds of proposition \ref{est-iter} are theoretical. It could therefore be of interest to see how they relate to the practically achieved $C^k$-norms of diffeomorphisms of the form used in section \ref{sec:numres} and  how these bounds are affected by for example the number of compositions. For simplicity, we stick to the one-dimensional case. In figure \ref{fig: norm estimates}, the value
\begin{equation}
    \frac{\|((\id_\Omega+f_L)\circ \cdots \circ (\id_\Omega+f_1))-\id_\Omega\|_{C^k}}{ \left(\sum_{\ell}^{L} \|f_\ell\|_{C^k}\right) e^{k\sum_{\ell}^{L} \|f_\ell\|_{C^k}} }
\end{equation}
is shown for lower-order derivatives with varying number of layers $L$ and number of basis functions $d$ per layer. For $k=1,\, 2$ the estimates are relatively good with values close to $60\%$ and $15\%$ of $M_k$ respectively. 

These estimates are largely dependent on the choice of basis functions, so the results may look different for another set of basis functions than the ones used here. Further details regarding how the values of figure \ref{fig: norm estimates} are computed can be found in appendix \ref{sect: appendix norm estimates}.

\begin{figure}[htpb]
    \centering
    \includegraphics[width=\textwidth]{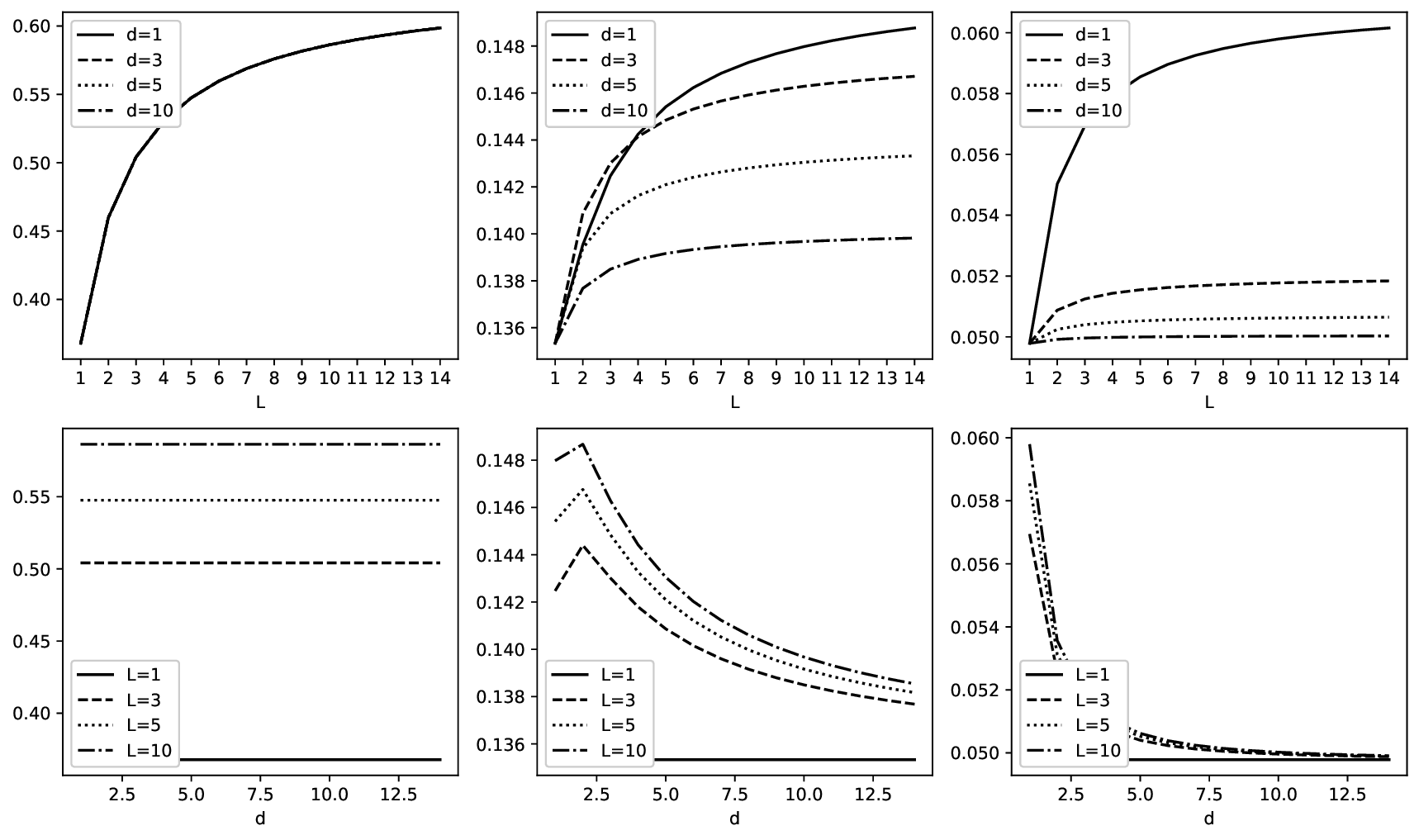}
\caption{Practical estimates of the $C^k$ bounding coefficient in equation \eqref{est-it-1} for $k=1, 2, 3$ (left to right).  (Top): Varying number of layers $L$. (Bottom): Varying number of basis functions per layer $M$.}
\label{fig: norm estimates}
\end{figure}        

\section{Implementation of the methods}\label{sect:methods}
This section describes the implementation of the approach outlined in section \ref{sect:deeplearning}. We first consider the case of curves and then the case of surfaces. In both cases, we need three parts: a discrete problem formulation, a choice of basis functions, and a projection operator to ensure that the Lipschitz constant of each layer is bounded by one. 

\subsection{Curves}\label{sect:methods-curves}
\subsubsection*{Problem Formulation}\label{sect:curves-problem-fomrulation}
%
%
Denote by $ W = \{\mathbf{w}^{\ell}\}_{\ell=1}^L \subset \mathbb{R}^d  $ a set of weight vectors $ \mathbf{w}^\ell $, and use the notation 
\begin{equation}
    \varphi(x; W) = \varphi_L(x; \mathbf{w}^{L}) \circ ... \circ \varphi_1(x; \mathbf{w}^{1}) 
    \label{eq: methods curves composition notation}
\end{equation}
to highlight the dependency of $ \varphi $ on \( W \), and $ \varphi_\ell $ on $ \mathbf{w}^{\ell}$, respectively.  Each of the layers are defined as residual-blocks 
\[
    \varphi_\ell(x, \mathbf{w}^\ell) = x + \sum_{n=1}^{M}w^\ell_n f_n(x)
\]
with basis function \(f_n\) as defined in \eqref{eq: curve basis}.
The goal is to find a set of vectors~$ W $ such that the function $ \varphi(\cdot, W) \in \text{Diff}^+(\Omega) $ minimizes the loss function \eqref{eq: optimization problem}. To approximate the shape distance, take a collection $ X = \{x_k\}_{k=1}^K \subset \Omega$ of linearly spaced points, and compute the mean squared error (MSE) between the $ q $-maps evaluated at these points. In other words, the loss function is approximated by the function
\begin{equation}
    E(W; X) = \frac{1}{K} \sum_{k=1}^{K} \lvert q(x_k) - \sqrt{\varphi'(x_k; W)} r(\varphi(x_k; W))\rvert^2.
    \label{eq: discrete-loss-func-curves}
\end{equation}
This function should be optimized with respect to the weights \( W \) under the constraints that $ \varphi(0, W) = 0 $, $ \varphi(1, W) = 1 $, and  $ \varphi'(\cdot: W) > 0 $. The optimization is handled by a BFGS-optimizer, with a projection step between the weight updates to ensure that the layers are invertible.\footnote{This approach has performed well for the experiments in this article, but we have found limited theory on the convergence of BFGS with a projection step. A large scale application of the reparametrization algorithm might gain from changing to a constrained optimization routine, or to another projected quasi-Newton algorithm with a better theoretical foundation.} The projection step is similar to the spectral normalization approach of \cite{behrmann2019}, and entails a scaling of the weight vectors such that the Lipschitz constants of each of the residual maps $ x\mapsto \sum_{n=1}^M w^\ell_n f_n(x) $ are bounded by one. 

\subsubsection*{Basis Functions} The tangent space of the diffeomorphism group on an interval consists of smooth functions that vanish on the boundary of the interval. A simple finite-dimensional basis which satisfies this relation is a truncated Fourier sine series.
However, the algorithm seems to perform better when the derivatives of the basis functions have similar magnitudes, and we scale the basis functions such that their derivatives are bounded by one. Therefore, we choose the basis of functions
\begin{equation}
    f_n(x) = \frac{\sin(n \pi x)}{n \pi}, \quad n = 1, ..., M
\label{eq: curve basis} 
\end{equation}
for $ M \in \N $.

\subsubsection*{Projection Operator} To ensure the invertibility of each layer \( \varphi_\ell \), we impose that the Lipschitz constant of each of the residual maps is smaller than one. For this purpose, denote by \( \mathcal{W}^\varepsilon \) the \textit{feasible set} of the weight vectors. To define this set explicitly, consider the upper estimate of the Lipschitz constants given by the inequalities
\begin{equation}
    \left\lvert\left(\sum_{n=1}^M w_n f_n(x)\right) - \left(\sum_{n=1}^M w_n f_n(y)\right)\right\rvert
    \leq \underbrace{\left( \sum_{n=1}^{M} \lvert w_n\rvert L_n \right)}_{:= L} \lvert x-y\rvert,
\label{eq: lipschitz bound}
\end{equation}
where $ L_n $ is the Lipschitz constant for the basis function $ f_n $.  For our choice of basis functions, we have that \(L_n = \sup_{x \in \Omega}|f_n'(x)| = 1\) for all $ \; n=1,\,..., M $. Thus \( L =\sum_{n=1}^{M} |w_n| L_n = \|\mathbf{w}\|_1 \). To ensure that the layers are diffeomorphisms, we require that $L$ is \textit{strictly} smaller than one. To achieve this, we define the feasible set
\begin{equation}
    \mathcal{W}^\varepsilon = \left\{ \mathbf{w}\in\mathbb{R}^M \; : 
    \;  \|\mathbf{w}\|_{1} \leq 1 -\varepsilon \right\}
\label{eq: curves feasible set} 
\end{equation}
for some small $\varepsilon$. If the optimization algorithm makes a weight update which causes a violation of these constraints, then it needs to be projected onto the feasible set. We define this projection using a simple scaling of the vector, defined by
\begin{equation}
\pi\colon \mathbb{R}^M \to \mathcal{W}^\varepsilon, \quad \pi(\mathbf{w}) = \frac{1-\varepsilon}{\max\{1 - \varepsilon, \|\mathbf{w}\|_1\}} \mathbf{w}.
\label{eq: curves projection map}
\end{equation}

\subsection{Surfaces}\label{sect:methods-surfaces}
\subsubsection*{Problem Formulation}\label{sect:surfaces-problem-fomrulation}
The problem formulation for reparametrization of surfaces is analogous to the case for curves, with the natural extensions to two dimensions. We approximate the continuous loss function by
\begin{equation}
    E(W; X) = \frac{1}{K^2} \sum_{i, j=1}^{K} \lvert q(\mathbf{x}_{i,j}) - \sqrt{J_\varphi(\mathbf{x}_{i, j}; W)} r\left( \varphi(\mathbf{x}_{i, j}; W) \right) \rvert^2,
\label{eq: discrete loss surfaces} 
\end{equation}
where $ X = \{ \mathbf{x}_{i, j}\,,\, i, j = 1,...,K\} \subset \Omega (\subset \mathbb{R}^2)$, and \(\varphi\) as in \eqref{eq: methods curves composition notation}. The loss function should be optimized with respect to the weights \( W \) under the constraints $ J_\varphi(\cdot: W) > 0 $ and $ \varphi(\partial M; W) = \partial M$.
\subsubsection*{Basis Functions}\label{sec:basis-elements}
The tangent space of the diffeomorphism group on $ \Omega = [0, 1]^2$ consists of smooth vector fields tangential to the boundary $ \partial \Omega $. We construct a basis for this space componentwise, as a tensor product basis of Fourier sine series in one direction and a full Fourier basis in the other. This choice was inspired by the approach adopted in  \cite{kurtek2012}. In the first component, these functions are given by the three families of functions,
\begin{subequations} \label{eq: basis vector fields}
\begin{align}
        \xi_k(x, y) &= \frac{\sin(\pi k x)}{\pi k},\\
        \eta_{k, l}(x, y) &= \frac{\sin(\pi k x) \cos(2\pi l y)}{\pi kl},\\
        \phi_{k, l}(x, y) &= \frac{\sin(\pi k x) \sin(2 \pi l y)}{\pi kl},
\end{align}
\end{subequations}
with corresponding basis functions \( \tilde \xi_k(x, y) = \xi_k(y, x) \), \( \tilde \eta_{k, l}(x, y) = \eta_{k, l}(y, x) \) and \(\tilde \phi_{k, l}(x, y) = \phi_{k, l}(y, x)\) for the second component. A few example basis vector fields are shown in figure \ref{fig: vectorfields}. The basis is truncated such that the coefficients $k, l = 1, ..., N$ for some fixed value of $N \in \N$.  The total number of basis functions is therefore $M = 2(2N^2 + N) $. In other words, the number of basis functions per layer in the network increases quadratically with $N$.
\begin{figure}[!htb] 
    \centering
    \includegraphics[width=\textwidth]{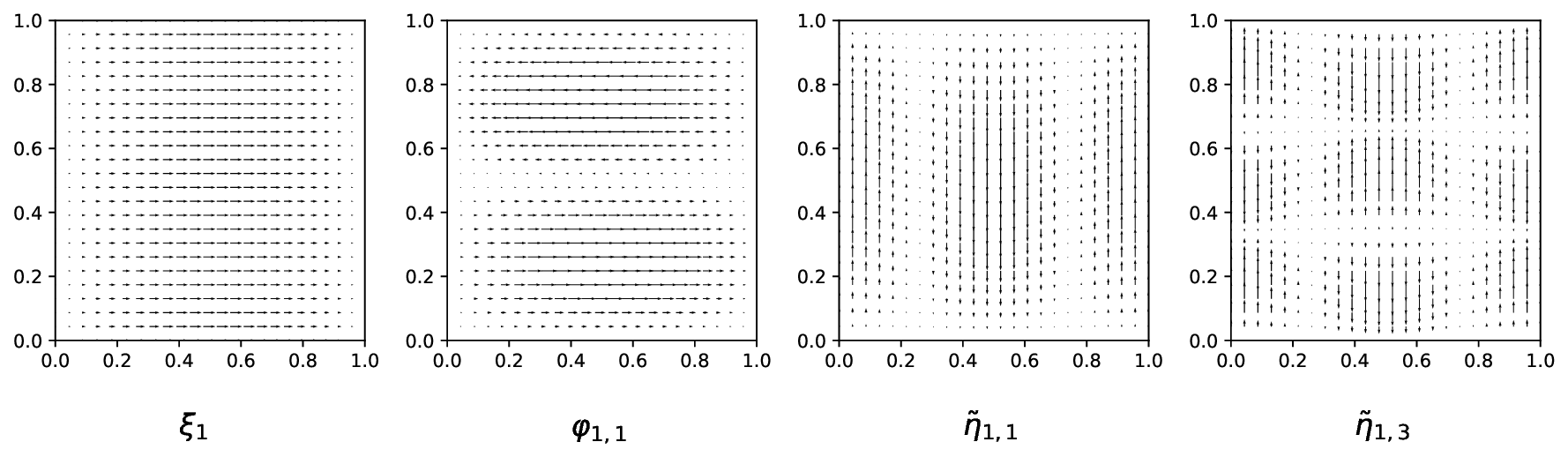}
    \caption{Four examples of basis elemets of $T_{\id}(\text{Diff}^+(\Omega))$, as described in \ref{sec:basis-elements}. The basis fields are constructed componentwise as products of trigonometric functions and point either purely in the \( x \)- or \( y \)-direction.}
    \label{fig: vectorfields}
\end{figure}

\subsubsection*{Projection Operator}
Similarly as in section \ref{sect:methods-curves}, the invertibility of each layer \( \varphi_\ell \) is ensured by bounding the Lipschitz constants of the residual maps by one. The projection operator is analogous to the one-dimensional case, and is based on the approximation \eqref{eq: lipschitz bound}. However, the Lipschitz constant $L_n$ of each of the basis vector fields is no longer simply equal to $1$. Instead, we make use of the following approximation:
\[
  L_n = \sup_{x \in \Omega} \|D f_n(x)\|_2 \leq  \sup_{x \in \Omega}\|D f_n(x)\|_F, \quad n=1,\,..., M ,
\]
where \( Df_n \) denotes the Jacobian matrix of \( f_n \), the symbol
$\|\cdot\|_2$ denotes the spectral norm, and $\|\cdot\|_F$ the Frobenius norm. According to this estimate, the Lipschitz constants for each of the three types of basis vector fields are bounded by
\begin{equation}
        L^\xi_{k} = 1, \quad L^\eta_{k, l} = \frac{\sqrt{k^2 + 2 l^2}}{kl}, \quad L^\phi_{k, l} = \frac{\sqrt{k^2 + 2 l^2}}{kl},
\end{equation}
and are found by taking the Frobenius norm of the Jacobian matrices of the basis functions in \eqref{eq: basis vector fields}.
Similar constants exist for the corresponding vector fields in the $y$-direction. Adopting the notations from the one-dimensional case, we define 
\begin{equation}
 \pi\colon \mathbb{R}^M \to \mathcal{W}^\varepsilon, \quad  \pi(\mathbf{w}) = \frac{1-\varepsilon}{\max\left\{1-\varepsilon, \left( \sum_{n=1}^{M} \lvert w_n\rvert  L_n \right)\right\}} \mathbf{w},
\end{equation}
which maps infeasible weight vectors to the feasible set.

\subsection{Shape Interpolation}\label{sect:methods-interpolation}
An interesting application of the reparametrization algorithm is to create realistic transitions between different shapes. The shape distance \( d_\Ess \) corresponds to the length of some geodesic between the two parametric shape representations in the pre-shape space $ \mathcal{P} $.
Each of the points along this geodesic corresponds to a parametric representation of some intermediate shape, providing a natural transformation between the original curves or surfaces. 

In the specific case of reparametrization of curves using the SRVT, we may compute a geodesic between their shapes. Assume for any $\tau \in (0, 1)$ that the curve $\tau \mathcal{R}(c_1)(x) + (1 - \tau) \mathcal{R}(c_2)(x) \neq 0 ,\,\forall\, x \in\Omega$. Then the map
\begin{equation}
    \tau \mapsto \mathcal{R}^{-1}(\tau \mathcal{R}(c_1) + (1 - \tau) \mathcal{R}(c_2))
\end{equation}
with
\begin{equation}
    \mathcal{R}^{-1}(q)(x) = \int_0^x q(\xi) \lvert q(\xi)\rvert \,d\xi
\end{equation}
defines a geodesic between the two curves $c_1$ and $c_2$.

The surface transforms, however, are not invertible. Therefore, we adopt an  alternative approach which seems to work well in many cases. Define the map
\[
    \gamma: \mathcal{P} \times \mathcal{P} \times [0, 1] \to C^\infty(\Omega, \mathbb{R}^n) , \quad \gamma(f_1, f_2, \tau) = \tau f_1 + (1 - \tau) f_2.
\] 
Note that for $f_1, f_2 \in \mathcal{P}$, the map $\tau \mapsto \gamma(f_1, f_2, \tau) $ defines a convex linear combination of two elements in $C^\infty(\Omega, \mathbb{R}^n)$ connecting the two surfaces.
However, when applied directly to the parametrized curves without performing the reparametrization, the corresponding path of this interpolation in shape-space is typically far from length-minimizing. 
To improve upon this shape transformation, we start by reparametrizing one of the shapes by finding $\varphi^* \in \argmin_{\varphi^+ \in \Diff(\Omega)} d_\mathcal{P}([f_1], [f_2 \circ \varphi]) $ and then apply the shape interpolation according to
\begin{equation}
    \tau \mapsto \gamma(f_1, f_2\circ\varphi^*, \tau).
\label{eq: shape interpolation} 
\end{equation}

\begin{figure}[hbpt]
    \centering
    \includegraphics[width=0.8\textwidth]{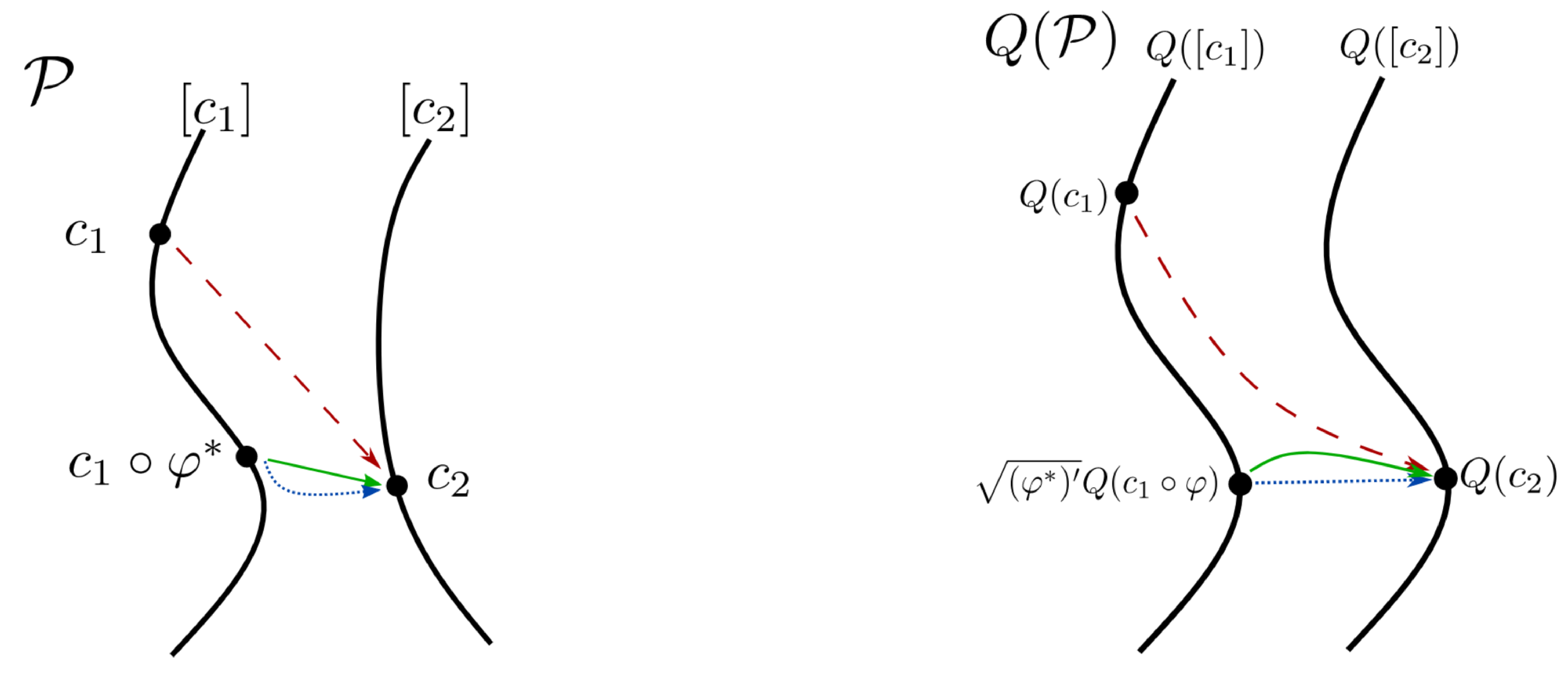}
    \caption{Different interpolation methods between two shapes and how these curves relate to different shape representations. The dashed red curve corresponds to the direct linear interpolation between the two original parametric shapes. The solid green curve corresponds to a geodesic, whereas the dotted blue line represents a linear interpolation after reparametrization given in \eqref{eq: shape interpolation}.}
    \label{fig: shape interpolation illustration} 
\end{figure}

Figure \ref{fig: shape interpolation illustration} illustrates how the different interpolation approaches compare to each other both in pre-shape space and in shape space. Figure~\ref{fig: curve interpolation} shows the different approaches applied to curves and figure \ref{fig: mnist linear interpolations} presents an example of the approach \eqref{eq: shape interpolation} applied to images.
Applied directly to the parametrized curves as in the left panel of \cref{fig: curve interpolation}, the intermediate curves of the transition between the two shapes does not follow what we intuitively consider a shortest path between the curves.
Instead of only changing the parts of the curves which differs between the two shapes, the entire curve is moved in the intermediate steps. 
This is in contrast to both the middle- and the right panel of \cref{fig: curve interpolation} where the upper half-circle is stationary, and only the lower half of the circle is moving between the shapes.

\begin{figure}[!htb]
    \centering
    \includegraphics[width=\textwidth]{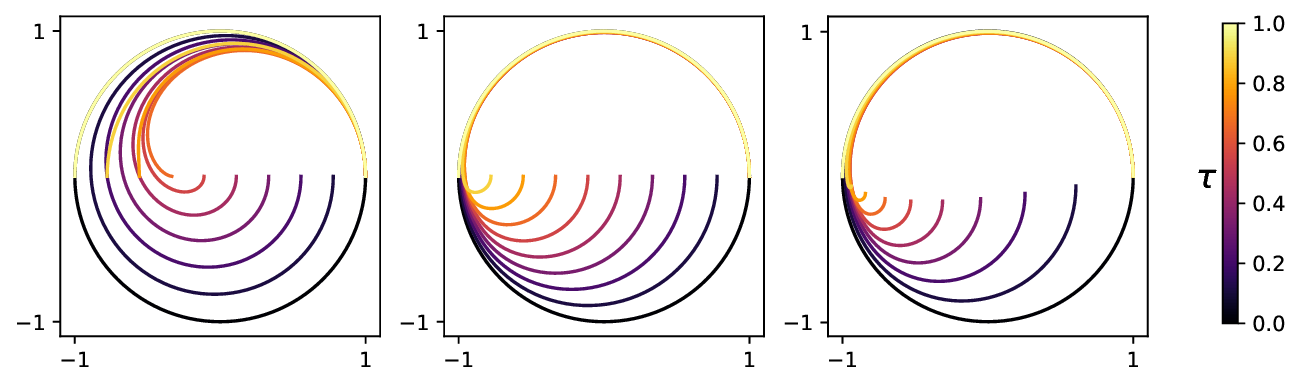}
    \caption{Comparison of the methods described in section \ref{sect:methods-interpolation} for interpolation between a half-circle and a circle. (Left) Direct linear interpolation, corresponding to the dashed red curve in \ref{sect:methods-interpolation}. (Middle) Linear interpolation after reparametrization, corresponding to the  dotted blue curve. (Right) Geodesic, corresponding to the solid green curve.}
    \label{fig: curve interpolation}
\end{figure}

\subsection{Notes on Implementation}\label{sect:implementation}
Since the structure of the diffeomorphisms that we are searching for is similar to that of a residual neural network, the weight optimization is analogous to the process of training a neural network. This is typically done by an iterative line search method such as gradient descent rather than Newton's method or other Quasi-Newton algorithms. However, since we are dealing with a deterministic problem, we have chosen to use the BFGS-algorithm (see e.g.\ \cite{nocedal2006}). The network is implemented using the PyTorch machine learning framework in Python\cite{paszke2019} which takes care of the computation of gradients and updates the weight vectors. By building our algorithms upon the PyTorch-framework, the algorithms easily achieve high performance, scalability and extensibility.


\FloatBarrier
\section{Numerical Experiments}\label{sec:numres}
\subsection{Curves}\label{sect:results-curves}
We start by comparing two different parametric curves representing the same shape. A simple way to find two such curves is by first selecting a curve $ c $, and reparametrize the curve with some diffeomorphism~$ \varphi $. By applying the reparametrization algorithm to the curves $ c_1 = c \circ\varphi $ and $ c_2 = c $, we expect to recover a diffeomorphism $ \psi \approx \varphi $ such that the distance $d_\mathcal{P}(c_1, c_2\circ\psi) \approx 0$. As a first example, consider the parametric curve and diffeomorphism given by

\begin{equation}
    \begin{gathered}
        c(t) = [\cos(2\pi t), \sin(4 \pi t)], \\
        \varphi(t) = \frac{\log(20t + 1)}{2\log(21)} + \frac{1 + \tanh(20(t - 0.5))}{4 \tanh(10)}.
    \end{gathered}
\label{eq: curves example} 
\end{equation}

The curves $ c $ and $ c \circ \varphi $ are shown in  figure~\ref{fig: same shape curves}. The dots along the curves represent linearly spaced points on the domain $\Omega = [0,1]$ for the curve parameter $t$. 
\Cref{fig: same shape curves} presents the results of applying the reparametrization algorithm to these curves using a network with $L=10$ layers and $M=10$ basis functions per layer. The resulting reparametrization $ \psi $ is visually indistinguishable from the true diffeomorphism $ \varphi $, and the components of the reparametrized curve $ c\circ \psi$ exactly match the components of the components of $ c \circ \varphi $.

\begin{figure}[htb]
    \centering
    \includegraphics[width=0.68\textwidth]{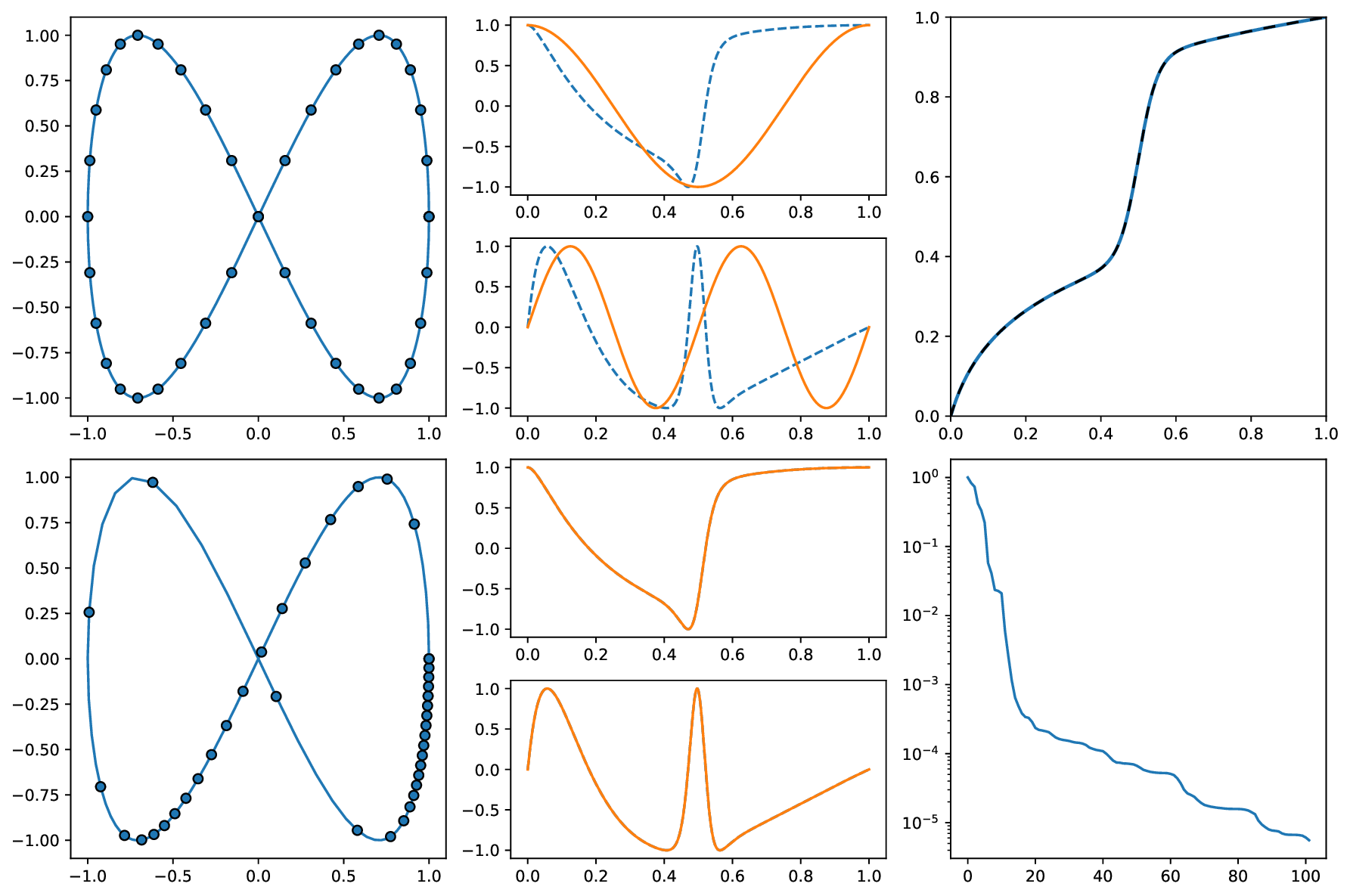}
    \caption{(Left) The curve $c$ to be reparametrized on top, and the target curve $c \circ \varphi$ below, defined in \eqref{eq: curves example}. (Middle) The component functions of the two curves $ c\circ\varphi $ (dotted blue) and $ c $ (solid orange) before, and after reparametrization. (Top right) The true reparametrization $\varphi$ (dotted black) compared to the reparametrization $\psi$ found by the algorithm. (Bottom right): The value of the loss function plotted against the iteration number of the weight updates, given relative to the initial error. Each iteration corresponds to one iteration of the BFGS algorithm including line-search.} 
    \label{fig: same shape curves}
\end{figure}

\FloatBarrier
\subsection{Surfaces}
Similarly as in section \ref{sect:results-curves}, the performance of the reparametrization algorithm is investigated by comparing surfaces representing the same shape. The results are presented in figures \ref{fig: surfaces same shape hyperboloid surfaces}, \ref{fig: surfaces same shape cylinder surfaces}. The coloring on the surfaces represents the local area scaling factor $ a_f  $ for the given surface. Once again, the algorithm does indeed find a diffeomorphism $ \psi \approx \varphi $. Moreover, based on figure~\ref{fig: surfaces convergence}, the convergence behavior of the reparametrization algorithm under an increasing number of basis functions and network layers is similar to the behavior already observed for curves.

The previous examples are using artificial surfaces defined by explicit parametric functions. In most practical applications, however, objects are represented by discrete observations. For example, images are typically stored as matrices representing image intensities. Such data must be transformed to parametrized surfaces before reparametrization.  For this purpose, we made use of a function\footnote{\href{https://pytorch.org/docs/stable/generated/torch.nn.functional.grid_sample.html}{torch.nn.functional.grid\_sample}} from the PyTorch library which implements a bicubic convolution algorithm, \cite{keys1981}. Since we are working with three-dimensional surfaces, it seems natural to represent single-channel images by their graphs. 

The fact that the optimal reparametrization is unknown when comparing surfaces representing different shapes, implies that it is difficult to properly verify the correctness of the algorithm. However, the shape interpolation method of section \ref{sect:methods-interpolation} provides a good alternative to validate the results; an optimal reparametrization of one surface relative to the other, should provide a smoother interpolation between the two. Figure \ref{fig: mnist linear interpolations} illustrates the interpolation between two images before and after reparametrization. Before reparametrization, the intermediate images clearly show elements from each of the images simultaneously, as one fades away and the other is revealed. After reparametrization the interpolation represents a smoother transition, with one of the surfaces being ``warped'' into the other.
{\textwidth=0.8\textwidth
\begin{figure}[hbtp]
    \centering
    \includegraphics[width=0.93\textwidth]{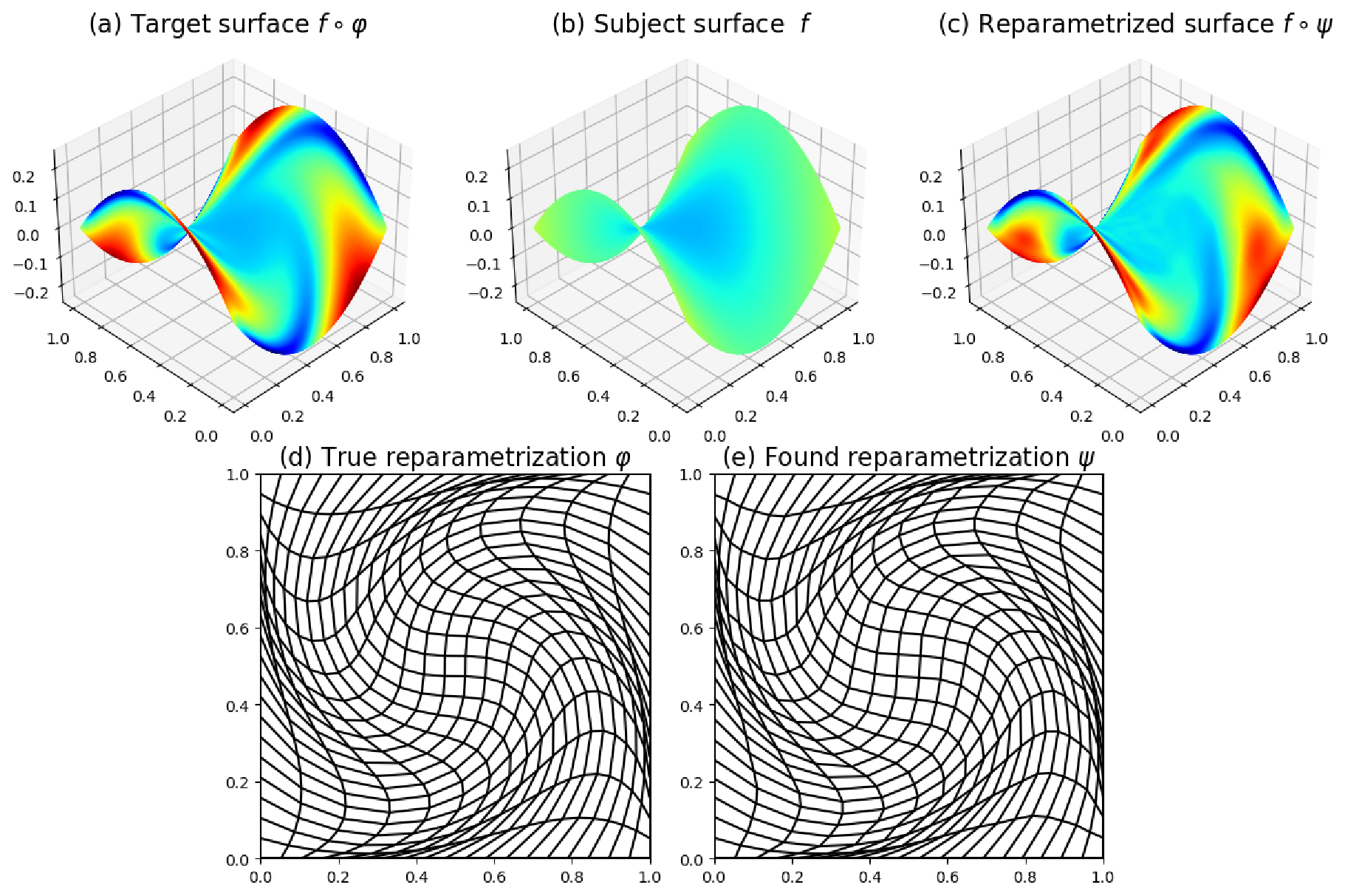}
    \caption{First example of reparametrization of a surface when compared to another parametric surface representing the same shape. The colors on the surfaces represent the area scaling factor of the parametrization. The target surface \( f \circ \varphi \) (a) have been created by reparametrizing the subject surface \( f \) (b) by some prescribed diffeomorphism \( \varphi \) (d). After applying the reparametrization algorithm to compare the two surfaces, the resulting reparametrization is shown in (e) and the reparametrized surface shown in (c).} 
    \label{fig: surfaces same shape hyperboloid surfaces}
\end{figure}

\begin{figure}[hbtp]
    \centering
    \includegraphics[width=0.93\textwidth]{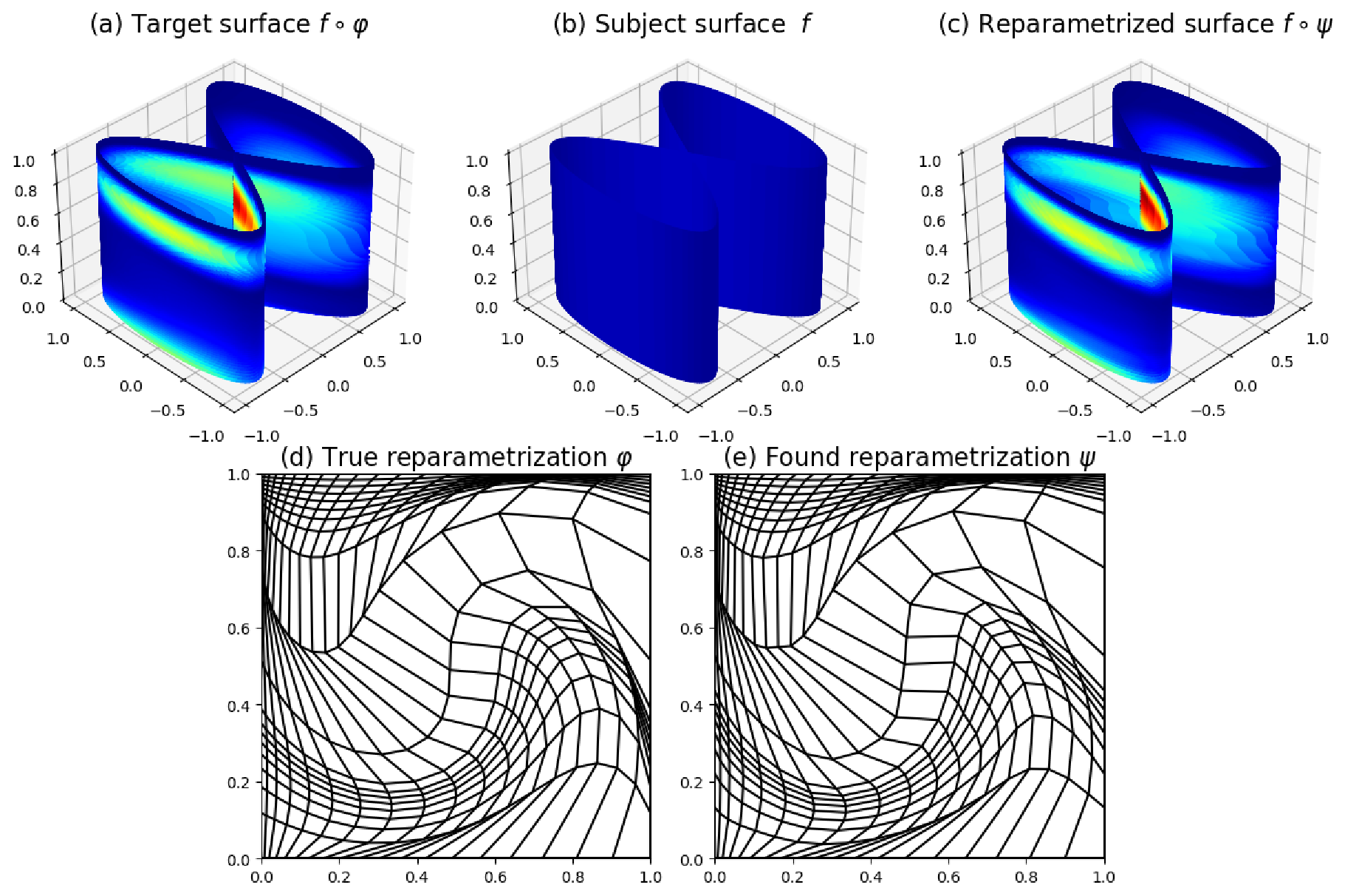}
    \caption{Second example of reparametrization of two surfaces representing the same shape such as in figure \ref{fig: surfaces same shape hyperboloid surfaces}, this time including an immersion and a more complex diffeomorphism involving both rotations and stretching.}
    \label{fig: surfaces same shape cylinder surfaces}
\end{figure}
}

\FloatBarrier

\begin{figure}[htbp]
    \centering
    \includegraphics[width=\textwidth]{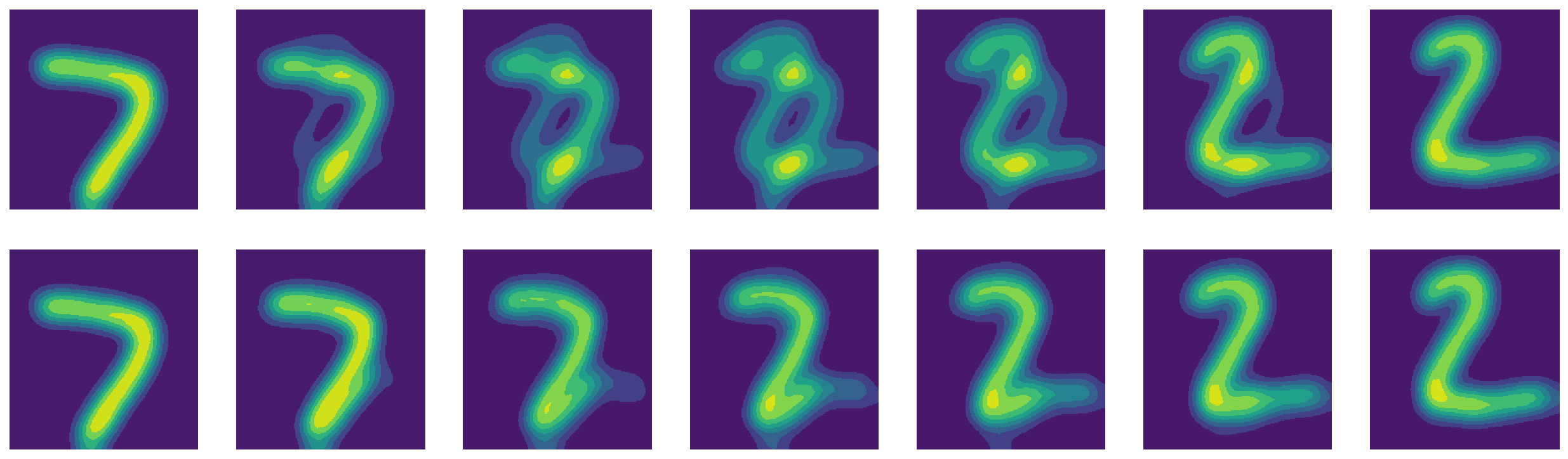} 
    \caption{Linear interpolation between two images from the MNIST digits dataset, before and after reparametrization of the surfaces. The first row corresponds to direct linear interpolation between the graphs, whereas the second row corresponds to linear interpolation after reparametrization as given in \eqref{eq: shape interpolation}.}
    \label{fig: mnist linear interpolations}
\end{figure}

\subsection{Comparing Network Sizes}
Figure \ref{fig: curves convergence} presents the results of the reparametrization algorithm for a varying number of layers and basis functions per layer in the network. The results represent the value of the loss function \eqref{eq: discrete-loss-func-curves} after reparametrization for the curves from the example in figure \ref{fig: same shape curves}. It seems that by keeping the number of layers low while increasing the number of basis functions, or vice versa, results in a relatively poor match between the functions. A combination of network depth and width is required to reduce the error to the smallest values.  

Figure \ref{fig: surfaces convergence} presents similar results based on the reparametrization of the surfaces in figure \ref{fig: surfaces same shape hyperboloid surfaces}. Note that $N$ represents the ``largest frequency'' of the basis functions, and not the number of basis functions. This number is much higher as pointed out in section \ref{sect:methods-surfaces}. In other words, a deep network will be able to achieve similar performance as a wide network, using much fewer coefficients. 

\begin{figure}[htb]
    \centering
    \includegraphics[width=0.9\textwidth]{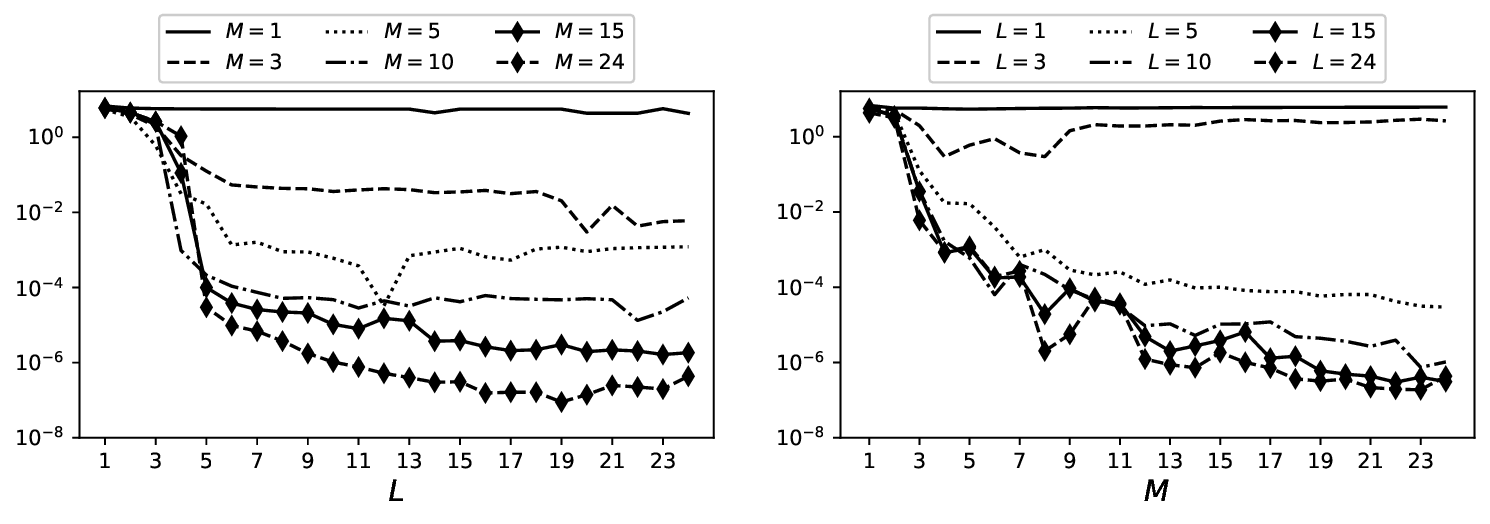}
    \caption{Final error after reparametrization of the curves in figure \ref{fig: same shape curves}, for varying number of layers $L$ and basis functions per layer $M$.}
    \label{fig: curves convergence}
\end{figure}
\begin{figure}[htb]
    \centering
    \includegraphics[width=0.9\textwidth]{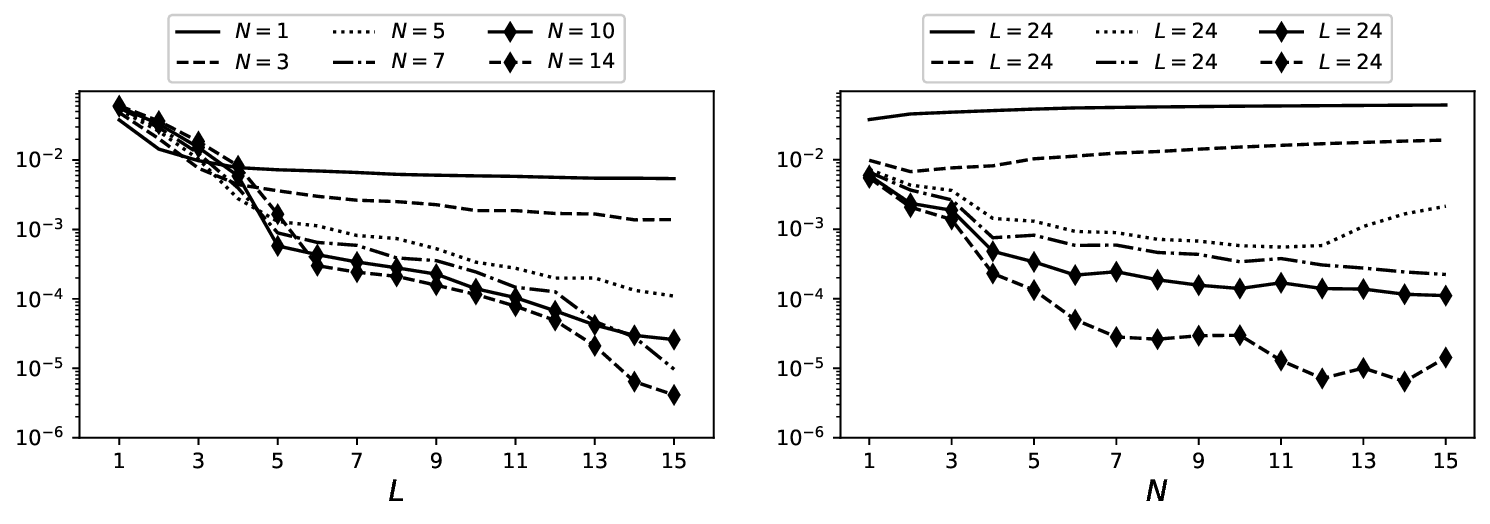}
    \caption{Final error after reparametrization of the surfaces in figure \ref{fig: surfaces same shape hyperboloid surfaces} for varying number of layers $L$ and the ``largest frequency'' $N$ of the basis functions The number of basis functions per layer \( M \) is in this case related to the frequency by $M = 2(2N^2 + N) $.}
    \label{fig: surfaces convergence}
\end{figure}

\FloatBarrier

\section{Discussion} \label{sec:discussion}
Based on \cref{fig: curves convergence} there seems to be a lower limit for how close to the optimal reparametrization one can get using this approach, as opposed to what one would expect based on the universal approximation properties presented in \cref{sect:Diffeos}.
The main reason behind this behavior seems to be the stopping criteria for the BFGS-optimizer. There are three criteria playing a role in the stopping behaviour: The termination tolerance on the gradient norm, the termination tolerance on the change of function value, or the maximal number of iterations in the optimizer.
By significantly lowering the termination tolerances and increasing the maximal number of iterations, it is possible to further reduce the errors for the larger networks.
This reduction, however, comes at a high computational cost relative to small improvements to the registration.

Due to the relationship of the algorithm with the gradient-descent algorithm, we have chosen to keep the structure and activation/basis functions from the original problem, rather than doing an extensive investigation into alternative network structures.
Because of the constraints of the diffeomorphism group, it does take some effort to switch network structures and activation functions we have considered this to be out of the scope of this article.
However, it is most certainly an interesting topic for future research. 

One important difference between a typical deep learning problem and the algorithm presented here, is that since we are registering a single pair of shapes, overfitting is desirable in our case --- at least under the assumption that the shapes are without noise. This assumption is justified since the noise from the original data will usually be handled by the method one uses to create a parametric curve/surface. This is typically done by interpolation or other regression methods with the possibility to add a regularizer or other preprocessing steps to reduce noise. An example is the bicubic interpolation method used for the MNIST-images in \cref{fig: mnist linear interpolations}. For future research, however, it would be interesting if one could incorporate the construction of the parametrization from a point collection into the network structure itself. In such a case, overfitting would indeed pose a problem, and typical regularization techniques would need to be applied. 

\section{Conclusion}
In this paper we proposed a method for the solution of optimization problems on the group of diffeomorphisms. The method is inspired by deep neural networks and implemented in PyTorch. An analysis of the universal approximation properties and a priori estimates of the norms of the obtained approximations are presented. A number of numerical results illustrate the merit of the proposed approach in applications of shape analysis.

\appendix
\section{Proof of Proposition \ref{est-iter}}\label{app:details}
In this appendix, we provide the details postponed in the proof of Proposition~\ref{est-iter}. Let us recall the statement for the reader's convenience:

\begin{numba}[Statement of Proposition \ref{est-iter}]
For all $k\in \N_0$ and all $L\in \N$,
\begin{eqnarray*}
    \lefteqn{\|((\id_\Omega+f_L)\circ \cdots \circ (\id_\Omega+f_1))-\id_\Omega\|_{C^k}}\qquad\qquad \\
    &\leq &  M_k\, e^{k\sum_{j=1}^L\|f_j\|_{C^k}}\,
    (\|f_1\|_{C^k}+\cdots+\|f_L\|_{C^k})
\end{eqnarray*}
holds for all $f_1,\ldots, f_L\in \Diff_{\partial}(\Omega)-\id_\Omega$ such that $\|f_1\|_{C^k}+\cdots+\|f_L\|_{C^k}\leq 1$.
\end{numba}

Recall that in Section \ref{sect:multcomp} we have already established the statement for all $L \in \N$ and $k \in \{0,1\}$. Thus all we need to prove is the statement for all $k\geq 2$.

\begin{proof}[Proof of Proposition {\rm\ref{est-iter}} in the case $k \geq 2$]
The proof is by induction.
Assume that $k\geq 2$ and assume that the asserted estimates already hold for $k-1$ in place of~$k$.
For $f,g\in\mathcal{V}$, applying Fa\'{a} di Bruno's formula to the first term in $f*g=(\id_\Omega+f)\circ (\id_\Omega+g)-\id_\Omega$, we get for all $x\in \Omega$ and $y_1,\ldots, y_k\in\wb{B}^E_1(0)$
\begin{eqnarray}
\lefteqn{d^k(f*g)(x,y_1,\ldots,y_k)}\qquad\qquad\notag \\
&=&
\sum_{j=2}^k\sum_{P\in P_{k,j}}
d^jf(h(x),d^{\lvert I_1 \rvert }h(x,y_{I_1}),\ldots,d^{\lvert I_j \rvert}h(x,y_{I_j}))\label{thesum}\\[.3mm]
&& + \,(\id_E+f'(x))(d^kg(x,y_1,\ldots,y_k)),\label{thesum2}
\end{eqnarray}
with $h:=\id_\Omega+g$. The norm of the summand in (\ref{thesum2}) is bounded by
\begin{equation}\label{thesum3}
(1+\|f'\|_{\infty,\op})\|g^{(k)}\|_{\infty,\op}
\leq e^{\|f\|_{C^k}}\|g\|_{C^k};
\end{equation}
this is $\leq M_k\, e^{k(\|f\|_{C^k}+\|g\|_{C^k})}\|g\|_{C^k}$.
The norm of a summand in (\ref{thesum}) is
\[
\leq \|f^{(j)}\|_{\infty,\op}\|h^{(\lvert I_1\rvert)}\|_{\infty,\op} \cdots\|h^{(\lvert I_j\rvert)}\|_{\infty,\op} \leq \|f\|_{C^k}\|h^{(\lvert I_1\rvert)}\|_{\infty,\op}
\cdots\|h^{(\lvert I_j\rvert )}\|_{\infty,\op}.
\]
For $a\in\{1,\ldots, j\}$, possible cases are that $\lvert I_a\rvert =1$; then
\[
\|h^{(\lvert I_a\rvert )}\|_{\infty,\op}=\|\id_E+g'\|_{\infty,\op}
\leq 1+\|g'\|_{\infty,\op}\leq e^{\|g\|_{C^k}}=e^{\lvert I_a \rvert\, \|g\|_{C^k}}.
\]
If $\lvert I_a \rvert \geq 2$, then 
\[
  \|h^{(\lvert I_a \rvert)}\|_{\infty,\op} =\|g^{(\lvert I_a\rvert)}\|_{\infty,\op}\leq \|g\|_{C^k}
\leq e^{\|g\|_{C^k}}\leq e^{\lvert I_a \rvert \, \|g\|_{C^k}}.
\]
As 
\[
  e^{\lvert I_1\rvert \,\|g\|_{C^k}}\cdots e^{\lvert I_j\rvert\,\|g\|_{C^k}}=e^{k\|g\|_{C^k}}\leq
M_{\lvert I_1\rvert }\cdots M_{\lvert I_j\rvert}e^{k(\|f\|_{C^k}+\|g\|_{C^k})},
\]
we get
we get
\begin{eqnarray*}
\|(f*g)^{(k)}\|_{\infty,\op} &\leq & \left(\sum_{j=2}^k\sum_{P\in P_{k,j}}
M_{\lvert I_1\rvert}\cdots M_{\lvert I_j \rvert }\right)\|f\|_{C^k}e^{k(\|f\|_{C^k}+\|g\|_{C^k})}\\[.5mm]
& & +  \,M_k\, e^{k(\|f\|_{C^k}+\|g\|_{C^k})}\|g\|_{C^k}\\[1.1mm]
& \leq & M_k\, e^{k(\|f\|_{C^k}+\|g\|_{C^k})}(\|f\|_{C^k}+\|g\|_{C^k}).
\end{eqnarray*}
Thus $\|f*g\|_{C^k}\leq M_k\, e^{k(\|f\|_{C^k}+\|g\|_{C^k})}(\|f\|_{C^k}+\|g\|_{C^k})$, whence (\ref{est-it-1}) holds for $L=2$. If we are given ${f_1,\ldots, f_L\in\mathcal{V}}$ with $L\geq 3$ and %
\begin{equation}\label{nowbdd}
\|f_1\|_{C^k}+\cdots+\|f_L\|_{C^k}\leq 1,
\end{equation}
let $f:=f_L$ and $g:=f_{L-1}*\cdots* f_1$. Using (\ref{thesum3}) and the inductive
hypothesis (the case of $L-1$ factors), we see that the right-hand side of~(\ref{thesum3}) is less or equal
\[
  \begin{split}
    e^{\|f_1\|_{C^k}}\|g\|_{C^k} &\leq e^{\|f_L\|_{C^k}} M_k\, e^{k(\|f_{L-1}\|_{C^k}+\cdots+\|f_1\|_{C^k})}(\|f_{L-1}\|_{C^k}+\cdots +\|f_1\|_{C^k})\\
    &\leq M_k\, e^{k\sum_{j=1}^L\|f_j\|_{C^k}}(\|f_{L-1}\|_{C^k}+\cdots +\|f_1\|_{C^k}).
  \end{split}
\]
With $h:=\id_\Omega+g=(\id_\Omega+f_{L-1})\circ\cdots\circ (\id_\Omega+f_1)$,
the norm of a summand in (\ref{thesum}) is bounded by
\[
\|f_L\|_{C^k}\|h^{(\lvert I_1 \rvert)}\|_{\infty,\op}
\cdots\|h^{(\lvert I_j \rvert)}\|_{\infty,\op}.
\]
For $a\in\{1,\ldots, j\}$, possible cases are that $\lvert I_a \rvert=1$; then
\begin{eqnarray*}
\|h^{(\lvert I_a \rvert)}\|_{\infty,\op} &=& \|((\id_\Omega+f_{L-1})\circ\cdots\circ(\id_\Omega+f_1))'\|_{\infty,\op}\\
&\leq & (1+\|f_{L-1}\|_{\infty,\op})\cdots(1+\|f_1\|_{\infty,\op})\\
&\leq &(1+\|f_{L-1}\|_{C^k})\cdots(1+\|f_1\|_{C^k}) \leq e^{\lvert I_a \rvert(\|f_{L-1}\|_{C^k}+\cdots+\|f_1\|_{C^k})}\\
&\leq & M_{\lvert I_a\rvert }e^{\lvert I_a \rvert(\|f_1\|_{C^k}+\cdots+\|f_L\|_{C^k})}.
\end{eqnarray*}
If $\lvert I_a \rvert \geq 2$, then
\begin{eqnarray*}
\| h^{(\lvert I_a \rvert )} \|_{\infty,\op}
& = & \| g^{(\lvert I_a\rvert )} \|_{\infty,\op}
\leq \|g\|_{C^{\lvert I_a \rvert}} \\
&\leq & M_{\lvert I_a \rvert}\, e^{\lvert I_a \rvert(\|f_{L-1}\|_{C^k}+\cdots+\|f_1\|_{C^k})}(\|f_{L-1}\|_{C^k}+\cdots+ \|f_1\|_{C^k}) \\
&\leq & M_{\lvert I_a \rvert} \, e^{\lvert I_a \rvert(\|f_1\|_{C^k}+\cdots+\|f_L\|_{C^k})},
\end{eqnarray*}
using (\ref{nowbdd}) for the final estimate.
As
\[
e^{\lvert I_1 \rvert(\|f_1\|_{C^k}+\cdots+\|f_L\|_{C^k})}\cdots
e^{\lvert I_j \rvert(\|f_1\|_{C^k}+\cdots+\|f_L\|_{C^k})}=e^{k(\|f_1\|_{C^k}+\cdots+\|f_L\|_{C^k})},
\]
we get
\begin{eqnarray*}
\lefteqn{\|(f_L*\cdots*f_1)^{(k)}\|_{\infty,\op}}\qquad\qquad\\
&\leq &
\left(\sum_{j=2}^k\sum_{P\in P_{k,j}}
M_{\lvert I_1 \rvert}\cdots M_{\lvert I_j \rvert}\right)\|f_1\|_{C^k}e^{k(\|f_1\|_{C^k}+\cdots+\|f_L\|_{C^k})}\\[.3mm]
& & \; +\, M_k\, e^{k(\|f_1\|_{C^k}+\cdots+\|f_L\|_{C^k})}(\|f_2\|_{C^k}+\cdots+
\|f_L\|_{C^k})\\[1mm]
& \leq & M_k\, e^{k\sum_{j=1}^L\|f_j\|_{C^k}}(\|f_1\|_{C^k}+\cdots+\|f_L\|_{C^k}),
\end{eqnarray*}
whence
$\|f_L*\cdots*f_1\|_{C^k}
\leq
M_k\, e^{k\sum_{j=1}^L\|f_j\|_{C^k}}(\|f_1\|_{C^k}+\cdots+\|f_L\|_{C^k})$.
This completes the proof. 
\end{proof}
\begin{rem}
Proposition~\ref{est-iter} was inspired by a recent result from the theory of evolution equations on infinite-dimensional
Lie groups (regularity theory). Consider a Lie group~$G$ modeled on a locally convex space~$E$, with neutral element~$e$.
Let $TG$ be the tangent bundle, $\cg:=T_eG$ and $G\times TG\to TG$, $(g,v)\mapsto g.v:=T L_g(v)$
be the natural left action of~$G$ on $TG$, which uses the tangent map $T L_g$ of the left translation $L_g\colon G\to G$, $x\mapsto gx$.
The Lie group~$G$ is called \emph{$C^0$-regular} if, for each continuous path $\gamma\colon [0,1]\to \cg$, there exists a $C^1$-map $\eta\colon [0,1]\to G$
with $\eta(0)=e$ and
\[
\eta'(t)=\eta(t).\gamma(t)\;\;\mbox{for all $\,t\in [0,1]$;}
\]
and moreover the time $1$-map $C([0,1],\cg)\to G$, $\gamma\mapsto\eta(1)$ is smooth (endowing the domain with the topology of uniform convergence);
see \cite{glockner2012,neeb2006}.
As shown by Hanusch~\cite{hanusch2017}, every $C^0$-regular Lie group is \emph{locally $\mu$-convex} in the sense of~\cite{gloeckner2017}.
Thus, for each open identity neighborhood $U\subseteq G$ and $C^\infty$-diffeomorphism $\phi \colon U\to V\subseteq E$ with $\phi(e)=0$, and each continuous seminorm
$q\colon E\to[0,\infty[$, there exists a continuous seminorm $p\colon E\to [0,\infty[$ such that
\[g_L\cdots g_2g_1\in U \]
and
\[
q(\phi(g_L\cdots g_1))\leq \, p(\phi(g_L))+\cdots+p(\phi(g_1))
\]
for all $L\in\N$ and $g_1,\ldots, g_L\in U$
such that $p(\phi(g_L))+\cdots+p(\phi(g_1))\leq 1$.
As shown in \cite{gloeckner2017}, the Lie group $\Diff_{\partial}(\Omega)$ is $C^0$-regular. Applying
the preceding fact to the $C^\infty$-diffeomorphism $\kappa\colon \id_\Omega+\mathcal{V}\to \mathcal{V}$, $g\mapsto g-\id_\Omega$
and the seminorm $q:=\|\cdot\|_{C^k}$ with a fixed $k\in \N_0$, we infer the existence
of a continuous seminorm $p$ on $C^\infty_{\partial}(\Omega, \mathbb{R}^d)$ such that
\[
\|((\id_\Omega+f_L)\circ\cdots\circ (\id_\Omega+f_1))-\id_\Omega\|_{C^k}\leq p(f_1)+\cdots+p(f_L)
\]
for all $L\in\N$ and $f_1,\ldots, f_L\in \mathcal{V}$ with $p(f_1)+\cdots+p(f_L)\leq 1$.
We may assume that $p$ is a multiple $r\|\cdot\|_{C^\ell}$ of $\|\cdot\|_{C^\ell}$ for some $\ell \geq k$, as these seminorms
define the topology on $C^\infty_{\partial}(\Omega ,\mathbb{R}^d)$. The point of Proposition~\ref{est-iter} is that we can always choose $\ell=k$,
and that it provides an explicit choice for the constant $r>0$.\\[2.3mm]
Previously, $\mu$-regularity has been discussed in quantitative form only for Banach-Lie groups,
where estimates for the Baker-Campbell-Hausdorff multiplication $*$ on the Banach-Lie algebra $\cg$ enable norms
$\|x_1*\cdots* x_m\|$ of products of elements $x_1,\ldots, x_m\in\cg$ to be estimated (see \cite{glockner2012}).
\end{rem}

\section{Practical upper bounds for \texorpdfstring{{\boldmath$C^k$}}{}-norm} \label{sect: appendix norm estimates}
This appendix describes the methods that were used for the values in figure \ref{fig: norm estimates}. The goal is to estimate how the practically achieved $C^k$-norms compare to the theoretical results of proposition \ref{est-iter}. For simplicity, we stick to the one-dimensional case. Since both sides of the inequality \ref{est-it-1} relies on the chosen vector fields \( f_\ell \), we start by reorganizing the inequality according to 
\begin{equation}
    \frac{\|((\id_\Omega+f_L)\circ \cdots \circ (\id_\Omega+f_1))-\id_\Omega\|_{C^k}}{ \left(\sum_{\ell}^{L} \|f_\ell\|_{C^k}\right) e^{k\sum_{\ell}^{L} \|f_\ell\|_{C^k}} } \leq M_k.
\end{equation}

\subsubsection*{Estimating \texorpdfstring{{\boldmath$C^k$}}{}-norms}
Let \( X=\{x_i\}_{i=1}^N \in\Omega \) be a collection of points tightly distributed throughout the domain \( \Omega \). Then we estimate the \( C^k \)-norm of a function \( f \in \Omega \) by $\|f\|_{C^k} = \|D^k f\|_{\text{op}, \infty} = \max_{x \in \Omega}\|D^k f(x)\|_{\text{op}} \approx \max_{\mathbf{x}_i \in X} \|D^k f(\mathbf{x_i})\|_{\text{op}}$
In the one-dimensional case, this reduces to
\begin{equation}
    \max_{x_i \in X} \left\lvert \frac{\partial^k f}{\partial x^k}(x_i)\right\rvert
\label{eq: norm estimate one dimensional} 
\end{equation}
where  \( X \) is typically chosen as an equidistant grid of points \( X = \left\{ i / N  \; \middle\vert \;  i=0, ..., N\ \right\}  \).

\subsubsection*{Choosing vector fields}
Consider vector fields of the form
\begin{equation}
    f_\ell(x) = \sum_{j=1}^{M} w_{\ell, j} \varphi_j(x),
\label{eq: vector field finite dimensional} 
\end{equation}
for basis functions as given in \eqref{eq: curve basis} and some weight vector \( \mathbf{w}_\ell \in \mathbb{R}^M \). Since different weight vectors might yield different results for the relative norm estimates, we use two different strategies for initializing the weight vectors:
\begin{enumerate}
    \item Deterministic vector of ones: \( w_{\ell, j}=1, \, j=1,...,M, \, \ell=1,...,L\), and
    \item Normal random sampling: \(w_{\ell, j} \sim \mathcal N(\mu=0, \sigma^2=1)\).
\end{enumerate}
The values of \eqref{eq: norm estimate one dimensional} presented in figure \ref{fig: norm estimates} are chosen as the largest attained value from strategy one and 500 runs of strategy 2.

\subsubsection*{Ensuring the assumptions hold}
The weight vectors cannot be chosen freely, but are assumed to adhere to a unit norm-sum constraint. To ensure that this condition holds, we will typically start out with an initial set of weight vectors, which will be scaled onto the feasible set. Let \( W = \{\mathbf{w}_\ell\}_{\ell=1}^L\) be a collection of  weight vectors, from which we create a set of vector fields \( f_\ell,\,\ell=1, ..., L \). Moreover, assume that for this set of vector fields 
\begin{equation}
    \sum_{\ell=1}^{L} \|f_\ell\|_{C^k} > 1.
\label{eq: norm sum violated} 
\end{equation}
The scaling will be performed in a way that is similar to the methods of section \ref{sect:methods-curves}; Find a scalar \( \alpha \in (0, 1] \) such that \( \alpha \sum_{\ell=1}^{L}\|f_\ell\|_{C^k} \leq 1 \), i.e. 
\[
   \alpha = \min\left\{1, \frac{1}{ \sum_{\ell=1}^{L} \|f_\ell\|_{C^k}}\right\}.
\]
and then choose $ \tilde W = \{\alpha \mathbf{w}_\ell\}_{\ell=1}^L  $.




\section*{Competing Interests}
The authors declare no competing interests related to the contents of the article.

\bibliographystyle{spmpsci}      
\bibliography{references.bib}   

\end{document}